\documentclass{lmcs}
\pdfoutput=1

\usepackage{lastpage}
\renewcommand{\dOi}{14(3:13)2018}
\lmcsheading{}{1--\pageref{LastPage}}{}{}%
{Feb.~15,~2018}{Aug.~31,~2018}{}

\usepackage{hyperref}

\usepackage{amsmath}
\usepackage{amssymb}
\usepackage{amsthm}

\usepackage{tikz}

\usepackage{catchfile}

\title{A Galois connection between Turing jumps and limits}

\author{Vasco Brattka}
\address{Faculty of Computer Science, Universit\"at der Bundeswehr M\"unchen, Germany and
Department of Mathematics and Applied Mathematics, University of Cape Town, South Africa}
\email{Vasco.Brattka@cca-net.de}

\begin{document} 



\def\AA{{\mathcal A}}
\def\BB{{\mathcal B}}
\def\CC{{\mathcal C}}
\def\DD{{\mathcal D}}
\def\EE{{\mathcal E}}
\def\FF{{\mathcal F}}
\def\GG{{\mathcal G}}
\def\HH{{\mathcal H}}
\def\II{{\mathcal I}}
\def\JJ{{\mathcal J}}
\def\KK{{\mathcal K}}
\def\LL{{\mathcal L}}
\def\MM{{\mathcal M}}
\def\NN{{\mathcal N}}
\def\OO{{\mathcal O}}
\def\PP{{\mathcal P}}
\def\QQ{{\mathcal Q}}
\def\RR{{\mathcal R}}
\def\SS{{\mathcal S}}
\def\TT{{\mathcal T}}
\def\UU{{\mathcal U}}
\def\VV{{\mathcal V}}
\def\WW{{\mathcal W}}
\def\XX{{\mathcal X}}
\def\YY{{\mathcal Y}}
\def\ZZ{{\mathcal Z}}


\def\bA{{\mathbf A}}
\def\bB{{\mathbf B}}
\def\bC{{\mathbf C}}
\def\bD{{\mathbf D}}
\def\bE{{\mathbf E}}
\def\bF{{\mathbf F}}
\def\bG{{\mathbf G}}
\def\bH{{\mathbf H}}
\def\bI{{\mathbf I}}
\def\bJ{{\mathbf J}}
\def\bK{{\mathbf K}}
\def\bL{{\mathbf L}}
\def\bM{{\mathbf M}}
\def\bN{{\mathbf N}}
\def\bO{{\mathbf O}}
\def\bP{{\mathbf P}}
\def\bQ{{\mathbf Q}}
\def\bR{{\mathbf R}}
\def\bS{{\mathbf S}}
\def\bT{{\mathbf T}}
\def\bU{{\mathbf U}}
\def\bV{{\mathbf V}}
\def\bW{{\mathbf W}}
\def\bX{{\mathbf X}}
\def\bY{{\mathbf Y}}
\def\bZ{{\mathbf Z}}


\def\IB{{\Bbb{B}}}
\def\IC{{\Bbb{C}}}
\def\IF{{\Bbb{F}}}
\def\IN{{\Bbb{N}}}
\def\IP{{\Bbb{P}}}
\def\IQ{{\Bbb{Q}}}
\def\IR{{\Bbb{R}}}
\def\IS{{\Bbb{S}}}
\def\IT{{\Bbb{T}}}
\def\IZ{{\Bbb{Z}}}

\def\IIB{{\Bbb{\mathbf B}}}
\def\IIC{{\Bbb{\mathbf C}}}
\def\IIN{{\Bbb{\mathbf N}}}
\def\IIQ{{\Bbb{\mathbf Q}}}
\def\IIR{{\Bbb{\mathbf R}}}
\def\IIZ{{\Bbb{\mathbf Z}}}


\def\ELSE{\quad\mbox{else}\quad}
\def\WITH{\quad\mbox{with}\quad}
\def\FOR{\quad\mbox{for}\quad}
\def\AND{\;\mbox{and}\;}
\def\OR{\;\mbox{or}\;}

\def\To{\longrightarrow}
\def\TO{\Longrightarrow}
\def\In{\subseteq}
\def\sm{\setminus}
\def\Inneq{\In_{\!\!\!\!/}}
\def\dmin{\mathop{\dot{-}}}
\def\splus{\oplus}
\def\SEQ{\triangle}
\def\DIV{\uparrow}
\def\INV{\leftrightarrow}
\def\SET{\Diamond}

\def\kto{\equiv\!\equiv\!>}
\def\kin{\subset\!\subset}
\def\pto{\leadsto}
\def\into{\hookrightarrow}
\def\onto{\to\!\!\!\!\!\to}
\def\prefix{\sqsubseteq}
\def\rel{\leftrightarrow}
\def\mto{\rightrightarrows}

\def\E{{\mathsf{{E}}}}
\def\B{{\mathsf{{B}}}}
\def\D{{\mathsf{{D}}}}
\def\J{{\mathsf{{J}}}}
\def\K{{\mathsf{{K}}}}
\def\L{{\mathsf{{L}}}}
\def\R{{\mathsf{{R}}}}
\def\T{{\mathsf{{T}}}}
\def\U{{\mathsf{{U}}}}
\def\W{{\mathsf{{W}}}}
\def\Z{{\mathsf{{Z}}}}
\def\w{{\mathsf{{w}}}}
\def\HP{{\mathsf{{H}}}}
\def\C{{\mathsf{{C}}}}
\def\Tot{{\mathsf{{Tot}}}}
\def\Fin{{\mathsf{{Fin}}}}
\def\Cof{{\mathsf{{Cof}}}}
\def\Cor{{\mathsf{{Cor}}}}
\def\Equ{{\mathsf{{Equ}}}}
\def\Com{{\mathsf{{Com}}}}
\def\Inf{{\mathsf{{Inf}}}}

\def\Tr{{\mathrm Tr}}
\def\Sierp{{\mathrm Sierpi{\'n}ski}}
\def\psisierp{{\psi^{\mbox{\scriptsize\Sierp}}}}
\def\cl{{\mathrm{{cl}}}}
\def\Haus{{\mathrm{{H}}}}
\def\Ls{{\mathrm{{Ls}}}}
\def\Li{{\mathrm{{Li}}}}

\def\LPO{\mathsf{LPO}}
\def\LLPO{\mathsf{LLPO}}
\def\WKL{\mathsf{WKL}}
\def\RCA{\mathsf{RCA}}
\def\ACA{\mathsf{ACA}}
\def\SEP{\mathsf{SEP}}
\def\BCT{\mathsf{BCT}}
\def\IVT{\mathsf{IVT}}
\def\IMT{\mathsf{IMT}}
\def\OMT{\mathsf{OMT}}
\def\CGT{\mathsf{CGT}}
\def\UBT{\mathsf{UBT}}
\def\BWT{\mathsf{BWT}}
\def\HBT{\mathsf{HBT}}
\def\BFT{\mathsf{BFT}}
\def\FPT{\mathsf{FPT}}
\def\WAT{\mathsf{WAT}}
\def\LIN{\mathsf{LIN}}
\def\B{\mathsf{B}}
\def\BF{\mathsf{B_\mathsf{F}}}
\def\BI{\mathsf{B_\mathsf{I}}}
\def\C{\mathsf{C}}
\def\CF{\mathsf{C_\mathsf{F}}}
\def\CN{\mathsf{C_{\IN}}}
\def\CI{\mathsf{C_\mathsf{I}}}
\def\CK{\mathsf{C_\mathsf{K}}}
\def\CA{\mathsf{C_\mathsf{A}}}
\def\WPO{\mathsf{WPO}}
\def\WLPO{\mathsf{WLPO}}
\def\MP{\mathsf{MP}}
\def\BD{\mathsf{BD}}
\def\Fix{\mathsf{Fix}}
\def\Mod{\mathsf{Mod}}

\def\w{\mathsf{w}}

\def\leqm{\mathop{\leq_{\mathrm{m}}}}
\def\equivm{\mathop{\equiv_{\mathrm{m}}}}
\def\leqT{\mathop{\leq_{\mathrm{T}}}}
\def\lT{\mathop{<_{\mathrm{T}}}}
\def\nleqT{\mathop{\not\leq_{\mathrm{T}}}}
\def\equivT{\mathop{\equiv_{\mathrm{T}}}}
\def\nequivT{\mathop{\not\equiv_{\mathrm{T}}}}
\def\leqwtt{\mathop{\leq_{\mathrm{wtt}}}}
\def\equiPT{\mathop{\equiv_{\P\mathrm{T}}}}
\def\leqW{\mathop{\leq_{\mathrm{W}}}}
\def\equivW{\mathop{\equiv_{\mathrm{W}}}}
\def\leqtW{\mathop{\leq_{\mathrm{tW}}}}
\def\leqSW{\mathop{\leq_{\mathrm{sW}}}}
\def\equivSW{\mathop{\equiv_{\mathrm{sW}}}}
\def\leqPW{\mathop{\leq_{\widehat{\mathrm{W}}}}}
\def\equivPW{\mathop{\equiv_{\widehat{\mathrm{W}}}}}
\def\leqFPW{\mathop{\leq_{\mathrm{W}^*}}}
\def\equivFPW{\mathop{\equiv_{\mathrm{W}^*}}}
\def\leqWW{\mathop{\leq_{\overline{\mathrm{W}}}}}
\def\nleqW{\mathop{\not\leq_{\mathrm{W}}}}
\def\nleqSW{\mathop{\not\leq_{\mathrm{sW}}}}
\def\lW{\mathop{<_{\mathrm{W}}}}
\def\lSW{\mathop{<_{\mathrm{sW}}}}
\def\nW{\mathop{|_{\mathrm{W}}}}
\def\nSW{\mathop{|_{\mathrm{sW}}}}
\def\leqt{\mathop{\leq_{\mathrm{t}}}}
\def\equivt{\mathop{\equiv_{\mathrm{t}}}}
\def\leqtop{\mathop{\leq_\mathrm{t}}}
\def\equivtop{\mathop{\equiv_\mathrm{t}}}

\def\bigtimes{\mathop{\mathsf{X}}}

\def\leqm{\mathop{\leq_{\mathrm{m}}}}
\def\equivm{\mathop{\equiv_{\mathrm{m}}}}
\def\leqT{\mathop{\leq_{\mathrm{T}}}}
\def\leqM{\mathop{\leq_{\mathrm{M}}}}
\def\equivT{\mathop{\equiv_{\mathrm{T}}}}
\def\equiPT{\mathop{\equiv_{\P\mathrm{T}}}}
\def\leqW{\mathop{\leq_{\mathrm{W}}}}
\def\equivW{\mathop{\equiv_{\mathrm{W}}}}
\def\nequivW{\mathop{\not\equiv_{\mathrm{W}}}}
\def\leqSW{\mathop{\leq_{\mathrm{sW}}}}
\def\equivSW{\mathop{\equiv_{\mathrm{sW}}}}
\def\leqPW{\mathop{\leq_{\widehat{\mathrm{W}}}}}
\def\equivPW{\mathop{\equiv_{\widehat{\mathrm{W}}}}}
\def\nleqW{\mathop{\not\leq_{\mathrm{W}}}}
\def\nleqSW{\mathop{\not\leq_{\mathrm{sW}}}}
\def\lW{\mathop{<_{\mathrm{W}}}}
\def\lSW{\mathop{<_{\mathrm{sW}}}}
\def\nW{\mathop{|_{\mathrm{W}}}}
\def\nSW{\mathop{|_{\mathrm{sW}}}}

\def\botW{\mathbf{0}}
\def\midW{\mathbf{1}}
\def\topW{\mathbf{\infty}}

\def\pol{{\leq_{\mathrm{pol}}}}
\def\rem{{\mathop{\mathrm{rm}}}}

\def\cc{{\mathrm{c}}}
\def\d{{\,\mathrm{d}}}
\def\e{{\mathrm{e}}}
\def\ii{{\mathrm{i}}}

\def\Cf{C\!f}
\def\id{{\mathrm{id}}}
\def\pr{{\mathrm{pr}}}
\def\inj{{\mathrm{inj}}}
\def\cf{{\mathrm{cf}}}
\def\dom{{\mathrm{dom}}}
\def\range{{\mathrm{range}}}
\def\graph{{\mathrm{graph}}}
\def\Graph{{\mathrm{Graph}}}
\def\epi{{\mathrm{epi}}}
\def\hypo{{\mathrm{hypo}}}
\def\Lim{{\mathrm{Lim}}}
\def\diam{{\mathrm{diam}}}
\def\dist{{\mathrm{dist}}}
\def\supp{{\mathrm{supp}}}
\def\union{{\mathrm{union}}}
\def\fiber{{\mathrm{fiber}}}
\def\ev{{\mathrm{ev}}}
\def\mod{{\mathrm{mod}}}
\def\sat{{\mathrm{sat}}}
\def\seq{{\mathrm{seq}}}
\def\lev{{\mathrm{lev}}}
\def\mind{{\mathrm{mind}}}
\def\arccot{{\mathrm{arccot}}}

\def\Add{{\mathrm{Add}}}
\def\Mul{{\mathrm{Mul}}}
\def\SMul{{\mathrm{SMul}}}
\def\Neg{{\mathrm{Neg}}}
\def\Inv{{\mathrm{Inv}}}
\def\Ord{{\mathrm{Ord}}}
\def\Sqrt{{\mathrm{Sqrt}}}
\def\Re{{\mathrm{Re}}}
\def\Im{{\mathrm{Im}}}
\def\Sup{{\mathrm{Sup}}}

\def\LSC{{\mathcal LSC}}
\def\USC{{\mathcal USC}}

\def\CE{{\mathcal{E}}}
\def\Pref{{\mathrm{Pref}}}

\def\Baire{\IN^\IN}

\def\TRUE{{\mathrm{TRUE}}}
\def\FALSE{{\mathrm{FALSE}}}

\def\co{{\mathrm{co}}}

\def\BBB{{\tt B}}

\newcommand{\SO}[1]{{{\mathbf\Sigma}^0_{#1}}}
\newcommand{\SI}[1]{{{\mathbf\Sigma}^1_{#1}}}
\newcommand{\PO}[1]{{{\mathbf\Pi}^0_{#1}}}
\newcommand{\PI}[1]{{{\mathbf\Pi}^1_{#1}}}
\newcommand{\DO}[1]{{{\mathbf\Delta}^0_{#1}}}
\newcommand{\DI}[1]{{{\mathbf\Delta}^1_{#1}}}
\newcommand{\sO}[1]{{\Sigma^0_{#1}}}
\newcommand{\sI}[1]{{\Sigma^1_{#1}}}
\newcommand{\pO}[1]{{\Pi^0_{#1}}}
\newcommand{\pI}[1]{{\Pi^1_{#1}}}
\newcommand{\dO}[1]{{{\Delta}^0_{#1}}}
\newcommand{\dI}[1]{{{\Delta}^1_{#1}}}
\newcommand{\sP}[1]{{\Sigma^\P_{#1}}}
\newcommand{\pP}[1]{{\Pi^\P_{#1}}}
\newcommand{\dP}[1]{{{\Delta}^\P_{#1}}}
\newcommand{\sE}[1]{{\Sigma^{-1}_{#1}}}
\newcommand{\pE}[1]{{\Pi^{-1}_{#1}}}
\newcommand{\dE}[1]{{\Delta^{-1}_{#1}}}

\def\QED{$\hspace*{\fill}\Box$}
\def\rand#1{\marginpar{\rule[-#1 mm]{1mm}{#1mm}}}

\def\BL{\BB}


\newcommand{\bra}[1]{\langle#1|}
\newcommand{\ket}[1]{|#1\rangle}
\newcommand{\braket}[2]{\langle#1|#2\rangle}

\newcommand{\ind}[1]{{\em #1}\index{#1}}
\newcommand{\mathbox}[1]{\[\fbox{\rule[-4mm]{0cm}{1cm}$\quad#1$\quad}\]}


\newenvironment{eqcase}{\left\{\begin{array}{lcl}}{\end{array}\right.}


\theoremstyle{plain}
\newtheorem{theorem}[thm]{Theorem}
\newtheorem{corollary}[thm]{Corollary}
\newtheorem{proposition}[thm]{Proposition}
\newtheorem{lemma}[thm]{Lemma}
\newtheorem{observation}[thm]{Observation}
\newtheorem{facts}[thm]{Fact}
\theoremstyle{definition}
\newtheorem{definition}[thm]{Definition}
\newtheorem{problem}[thm]{Problem}
\newtheorem{assumption}[thm]{Assumption}
\newtheorem{question}[thm]{Question}
\newtheorem{example}[thm]{Example}
\newtheorem{convention}[thm]{Convention}

\subjclass{[{\bf Theory of computation}]:  Logic; [{\bf Mathematics of computing}]: Continuous mathematics.}

\begin{abstract}
Limit computable functions can be characterized by 
Turing jumps on the input side or limits on the output side.
As a monad of this pair of adjoint operations we obtain a problem
that characterizes the low functions and dually to this another problem
that characterizes the functions that are computable relative to 
the halting problem. Correspondingly, these two classes are the
largest classes of functions that can be pre or post composed to
limit computable functions without leaving the class of limit computable
functions. We transfer these observations to the lattice of represented
spaces where it leads to a formal Galois connection. 
We also formulate a version of this result for computable metric spaces.
Limit computability and computability relative to the halting problem
are notions that coincide for points and sequences, but even restricted to continuous functions
the former class is strictly larger than the latter. 
On computable metric spaces we can characterize the functions that are
computable relative to the halting problem as those functions that are
limit computable with a modulus of continuity that is computable relative to the halting problem.
As a consequence of this result we obtain, for instance, that Lipschitz continuous functions that
are limit computable are automatically computable relative to the halting problem.
We also discuss 1--generic points as the canonical points of continuity of
limit computable functions, and we prove that restricted to these points
limit computable functions are computable relative to the halting problem.
Finally, we demonstrate how these results can be applied in computable analysis.
\end{abstract}

\maketitle

\section{Introduction}

Limit computable functions have been studied for a long time. In computability theory
limits appear, for instance, in the form of Shoenfield's limit lemma~\cite{Sho59}.
In algorithmic learning theory they have been introduced by Gold~\cite{Gol65}.
Later, limit computable functions have been studied by  Wagner~\cite{Wag76,Wag77}, who
calls them \emph{Turing operators of the first kind} (and he attributes this class to Fre{\u\i}vald~\cite{Fre74}). 
Wagner discusses composition and also normal form theorems.
In computable analysis limit computations were studied by Freund~\cite{Fre83}
who introduced Turing machines that can revise their output as well as non-deterministic Turing machines.
Later, similar machines were systematically studied by Ziegler~\cite{Zie07,Zie07a}.
Limit computable real numbers were analyzed by Freund and Staiger~\cite{FS96} and by
Zheng and Weihrauch~\cite{ZW01} and others.
On the other hand, Ho~\cite{Ho96,Ho99} studied functions in analysis that are 
computable relative to the halting problem.

Many of these results are scattered in the literature, and it remains somewhat unclear 
how all these results are related. 
Some clarity can be brought into the picture if one starts with the notion of a limit
computable function and develops it systematically. 
It turns out that at the heart of such a development there is an adjoint situation
 that is related to the fact that limit computations can either
be described with limits on the output side or with Turing jumps on the input side.
Shoenfield's limit lemma is the non-uniform correlate of this observation.
From this perspective functions computable relative to the halting problem
and low functions are natural notions that should be studied alongside
with limit computable functions, since they constitute the largest classes
of functions that can be post or pre composed with limit computable functions
without leaving the class of limit computable functions.
We will develop these notions systematically in sections~\ref{sec:limit-computability}
and \ref{sec:computability-halting-problem}.\footnote{We note that an unpublished draft of some of the material presented here
was circulated since 2007 and it has influenced some of the
references, such as joint work by the author with de Brecht and Pauly~\cite{BBP12}.
The material presented in section~\ref{sec:computability-halting-problem}
and everything based on it has only been developed recently.}
Together with these notions we also study $1$--generic points, as these
are the canonical points of continuity of limit computable functions.

In section~\ref{sec:represented-spaces} we transfer these results to the setting 
of represented spaces and hence to other data types. The crucial notion here is the
notion of a jump of a representation that was introduced by Ziegler~\cite{Zie07,Zie07a}
and further generalized by de Brecht~\cite{dBre14}.
On the level of represented spaces the adjoint situation between limits and 
Turing jumps can formally be expressed as a Galois connection. 
We also study how the jump of a represented space interacts with 
other constructions on represented spaces such as building products and
exponentials. 

In section~\ref{sec:metric} we transfer our results to the setting of computable
metric spaces. The limit normal form for limit computable functions can 
easily be expressed using limits in computable metric spaces and for the jump
normal form one can use a general jump operation on computable metric spaces.
We also transfer a characterization of $1$--generic points to 
the level of computable metric spaces. 

In section~\ref{sec:limit-halting} we continue to study mostly
computable metric spaces and specifically the relation between
limit computable functions and functions computable relative to the 
halting problem. One characterization shows that functions that 
are computable relative to the halting problem are exactly those
continuous limit computable functions that admit a global modulus
of continuity that is computable relative to the halting problem.
We also provide several examples of functions of different type that show that
not all continuous limit computable functions are computable
relative to the halting problem. 
However, we can provide oracle classes that are sufficient to compute
limit computable functions that are continuous, uniformly continuous
and Lipschitz continuous, respectively. 
In particular for Lipschitz continuous functions limit computability
and computability with respect to the halting problem coincide.

We close this article with section~\ref{sec:applications} 
in which we demonstrate how some known results can 
easily be derived using the techniques provided in this article.
We also include some simple new applications.
The discussed examples also highlight
connections between results that are seemingly unrelated or
scattered in other sources.

\section{Limit computability}
\label{sec:limit-computability}

In this section we are going to characterize limit computable functions in different ways.
We start with recalling the definition. Limit computable functions are defined on Turing
machines that have a two-way output tape, and these machines are allowed to revise their output.
 
\begin{definition}[Limit computable functions]
A function $F:\In\IN^\IN\to\IN^\IN$ is called \emph{limit computable},
if there exists a Turing machine $M$ that upon any input $p\in\dom(F)$ computes 
$F(p)$ on its two-way output tape in the long run
such that on any output position $n\in\IN$ after a finite number of steps the output is 
$F(p)(n)$ and will not be changed anymore. 
\end{definition}

If the Turing machine does only change the entire output (over all $n\in\IN$) finitely many times
for each fixed input $p$, then $F$ is called \emph{computable with finitely many mind changes}.
This is a strengthening of the notion of limit computability.

It is clear that all computable maps are limit computable and limit 
computable maps are obviously closed under restriction.
The additional flexibility to revise the output tape gives
additional power to limit machines that ordinary Turing machines
do not have. As examples we mention some discontinuous functions that
are not computable, but limit computable.
For the definition we use tupling functions. We define
$\langle n,k\rangle:=\frac{1}{2}(n+k)(n+k+1)+k$,
$\langle p,q\rangle(2n):=p(n)$ and $\langle p,q\rangle(2n+1):=q(n)$,
$\langle p_0,p_1,p_2,....\rangle\langle n,k\rangle:=p_n(k)$, and
$\langle n,p\rangle:= np$
for all $p,q,p_i\in\IN^\IN$ and $n,k\in\IN$.
By $\widehat{n}=nnn...\in\IN^\IN$ we denote the constant sequence with value $n\in\IN$.

\begin{example}[Limit computable maps]
\label{example:limit-computable}
The following maps are
limit computable but not continuous and hence not computable:
\begin{enumerate}
\item The \emph{equality test for zero}
\[\E:\IN^\IN\to\IN^\IN,p\mapsto\left\{\begin{array}{ll}
  \widehat{1} & \mbox{if $p=\widehat{0}$}\\
  \widehat{0} & \mbox{otherwise}
\end{array}\right.\]
\item The \emph{limit map} 
\[\lim:\In\IN^\IN\to\IN^\IN,\langle p_0,p_1,p_2,...\rangle\mapsto\lim_{i\to\infty}p_i.\]
\item The \emph{Turing jump} $\J:\IN^\IN\to\IN^\IN$ by
\[\J(p)(i):=\left\{\begin{array}{ll}
1 & \mbox{if Turing machine $i$ halts upon input $p$}\\
0 & \mbox{otherwise}
\end{array}\right.
\]
for all $p\in\IN^\IN$ and $i\in\IN$.
\end{enumerate}
The function $\E$ is even computable with finitely many mind changes.
\end{example}

By $\lim_\Delta$ we denote the restriction of $\lim$ to convergent sequences with respect to the discrete topology on $\IN^\IN$.
In some sense one can say that the limit map and $\J$ are not only 
some limit computable maps, but prototypes of limit computable maps.
Before we can prove a formal version of this statement, we first discuss compositions of computable
and limit computable maps that turn out to be limit computable.

\begin{proposition}[Composition]
\label{prop:composition-limit}
If $F:\In\IN^\IN\to\IN^\IN$ is limit computable and $G:\In\IN^\IN\to\IN^\IN$ 
is computable, then $F\circ G$ and $G\circ F$ are both limit computable. 
\end{proposition}
\begin{proof}
Let $M_F$ be a limit machine computing $F$ and $M_G$ an ordinary Turing machine computing $G$.
Then a limit machine $M$ computing $F\circ G$ can just be obtained by composing the two machines $M_F$ and $M_G$
in the straightforward way.

A limit machine $M'$ for $G\circ F$ can also be constructed by composing the machines $M_F$ and $M_G$.
However, the composition has to be done such that if $M_F$ changes the content of some output cell $n$,
then the computation of $M_G$ has to be restarted at time step $n$ of its program (which is sufficient to guarantee that $M_G$ has
not yet seen the content of cell $n$). Since the output of $M_F$ of any finite length eventually stabilizes and
$M_G$ uses a one-way output tape, it follows that $M'$ produces a converging output
in this way for any input $p\in\dom(G\circ F)$. 
\end{proof}

It is easy to see that the first statement of the proposition can even be strengthened 
in the following way.

\begin{proposition}[Composition with finitely many mind changes]
\label{prop:composition-finite}
If $F:\In\IN^\IN\to\IN^\IN$ is limit computable and $G:\In\IN^\IN\to\IN^\IN$ is computable with finitely many mind changes,
then $F\circ G$ is limit computable.
\end{proposition}

For the composition $G\circ F$ of functions as in Proposition~\ref{prop:composition-finite} the 
strategy of the corresponding proof of Proposition~\ref{prop:composition-limit} does not work in general.
In fact, it is not too difficult to construct a counterexample. 

\begin{example}[Composition]
\label{ex:composition}
The equality test $\E$ is computable with finitely many mind changes, and the limit map $\lim$ is limit computable, 
but $\E\circ\lim$ is not limit computable.
\end{example}

We leave the proof to the reader.
Now we provide a very useful characterization of limit computable functions. 
This result is also implicit in the work of Wagner~\cite{Wag76}.

\begin{theorem}[Limit normal form]
\label{thm:limit}
A function $G:\In\IN^\IN\to\IN^\IN$ is limit computable if and only
if there exists a computable function $F:\In\IN^\IN\to\IN^\IN$, such that
$G=\lim\circ\; F$.
\end{theorem}
\begin{proof}
By Example~\ref{example:limit-computable} $\lim$ is limit computable.
It follows from Proposition~\ref{prop:composition-limit} that if $F$ is computable,
then $\lim\circ\; F$ is limit computable.

Conversely, if $G$ is limit computable, then we construct a one-way
output Turing machine $M_F$ that computes a function $F:\In\IN^\IN\to\IN^\IN$.
This machine $M_F$ internally simulates a limit machine $M_G$ for $G$ on input $p$. The machine $M_F$ successively produces an output $\langle q_0,q_1,...\rangle$ in
steps $\langle i,j\rangle=0,1,2,...$. 
In step $\langle i,j\rangle$ machine $M_F$ simulates $M_G$ for $i$ time steps beyond the point where $M_G$ has filled its $j$--th output position 
for the first time. Machine $M_F$ writes the resulting content of the simulated output of $M_G$ in the $j$--th position into its own output $q_i(j)$.  
If the simulated $j$--th position eventually stabilizes, then $\lim_{i\to\infty}q_i(j)$
exists and coincides with the final value of the $j$--th position $G(p)(j)$.
Thus the machine $M_F$ computes a function $F:\In\IN^\IN\to\IN^\IN$
with $G=\lim\circ\; F$. Here $F$ is restricted to $\dom(G)$.
\end{proof}

We note that the limit normal form theorem allows us to study limit computable
functions in terms of ordinary computable functions and without considering limit machines. 
The particular output behavior of limit machines can be expressed directly using the limit map.
We also obtain the following characterization of limit computable functions as pointwise limits
of a sequence of computable functions.

\begin{corollary}[Pointwise limit]
\label{cor:pointwise}
A function $G:\In \IN^\IN\to \IN^\IN$ is limit computable if and only if there is a computable sequence $(F_n)_{n\in\IN}$
of computable functions $F_n:\In\IN^\IN\to\IN^\IN$ with $\dom(G)\In\dom(F_n)$ and such that $\lim_{n\to\infty}F_n(p)=G(p)$
for all $p\in\dom(G)$.
\end{corollary}

Theorem~\ref{thm:limit} together with Proposition~\ref{prop:composition-limit} yield the following corollary.

\begin{corollary}
\label{cor:completeness-lim}
For any computable function $F:\In\IN^\IN\to\IN^\IN$ there exists a computable function
$G:\In\IN^\IN\to\IN^\IN$ such that $\lim\circ\, G=F\circ\lim$.
\end{corollary}

Corollary~\ref{cor:completeness-lim} also holds with continuity instead of computability in both occurrences. 
We note that the computable functions $F$ do not constitute the largest class of functions for which
there is a computable $G$ with $\lim\circ G=F\circ\lim$. 
In Corollary~\ref{cor:computability-halting-problem} we extend this result to functions that are computable relative to the halting problem.
Maps that satisfy the same property as the limit map in Corollary~\ref{cor:completeness-lim} have been called
\emph{jump operators} \cite{BBP12,dBre14} or \emph{transparent} \cite{BGM12}. 

\begin{definition}[Transparency]
A function $T:\In\IN^\IN\to\IN^\IN$ is called \emph{transparent} if for every computable $F:\In\IN^\IN\to\IN^\IN$
there exists a computable function $G:\In\IN^\IN\to\IN^\IN$ such that $T\circ G=F\circ T$.
\end{definition}

It is easy to see that not all functions are transparent, but the class of transparent functions is reasonably large.

\begin{proposition}[Transparency]
The class of transparent functions $T:\In\IN^\IN\to\IN^\IN$ forms a monoid with respect to composition,
i.e., the identity is transparent, and transparent functions are closed under composition.
\end{proposition}

It is clear that limit computable maps do not necessarily map computable inputs to computable outputs.
However, they map computable inputs to limit computable outputs.

\begin{definition}[Limit computable points]
A point $p\in\IN^\IN$ is called \emph{limit computable}, if there exists
a computable sequence $(p_i)_{i\in\IN}$ in $\IN^\IN$ such that $p=\lim_{i\to\infty}p_i$.
\end{definition}

It follows directly from the limit normal form theorem (Theorem~\ref{thm:limit}) and 
the fact that computable maps map computable inputs to computable outputs 
that limit computable maps map computable inputs to limit computable outputs.

\begin{corollary}
If $F:\In\IN^\IN\to\IN^\IN$ is limit computable and $p\in\dom(F)$ is computable,
then $F(p)$ is limit computable.
\end{corollary}

It also follows from the limit normal form theorem (Theorem~\ref{thm:limit}) that a point $p\in\IN^\IN$
is limit computable if and only if it is the value of a constant limit computable
map $c:\IN^\IN\to\IN^\IN$. 

We now study the Turing jump $\J$, and we first prove that its inverse is computable.
Even though $\J$ is not continuous, it is injective and has a computable inverse.

\begin{proposition}[Jump inversion]
\label{prop:jump-inversion}
The Turing jump operator $\J$ is injective, and its partial inverse $\J^{-1}:\In\IN^\IN\to\IN^\IN$ is computable.
\end{proposition}
\begin{proof}
There exists a computable function $r:\IN\to\IN$ such that the Turing machine with code $r\langle n,k\rangle$
halts upon input $p\in\IN^\IN$ if and only if $p(n)=k$. The function
$F:\In\IN^\IN\to\IN^\IN$ defined by
\[F(q)(n):=\min\{k\in\IN:qr\langle n,k\rangle=1\}\]
for all $q\in\IN^\IN$, $n\in\IN$ and $\dom(F):=\{q\in\IN^\IN:(\forall n)(\exists k)\;qr\langle n,k\rangle=1\}$
is computable. We obtain
\[F\circ\J(p)(n)=\min\{k\in\IN:\J(p)(r\langle n,k\rangle)=1\}=\min\{k\in\IN: p(n)=k\}=p(n)\]
for all $p\in\IN^\IN$ and $n\in\IN$, i.e., $F\circ\J=\id$. 
This implies that $\J$ is injective and $\J^{-1}=F|_{\range(\J)}$ is computable.
\end{proof}

We also write $p':=\J(p)$, and we call $p'$ the \emph{Turing jump} of $p$, and we call $0':=\J(\widehat{0})$ the \emph{halting problem}.
We note that the Turing jump operator considered as a map on Turing degrees is not injective.\footnote{There are $p\nequivT q$ such that $p'\equivT q'$.}
We mention that Proposition~\ref{prop:jump-inversion} also implies that any point can be reduced to its Turing jump,
i.e., $p\leqT p'$ (of course, this reduction is known to be strict, which can be proved by an easy diagonalization argument).

It is perhaps surprising that there is a dual version of the limit normal form theorem
that characterizes limit computation by an input modification with the help of Turing jumps
instead of an output modification with limits. 
This shows that in some sense Turing jumps and topological limits are adjoint to each other.
An analogous result can be found in the work of Wagner~\cite{Wag76}.

\begin{theorem}[Jump normal form]
\label{thm:jump}
A function $F:\In\IN^\IN\to\IN^\IN$ is limit computable, 
if and only if there exists a computable function $G:\In\IN^\IN\to\IN^\IN$
such that 
$F=G\circ\J$.
\end{theorem}
\begin{proof}
By Example~\ref{example:limit-computable} the Turing jump operator $\J:\IN^\IN\to\IN^\IN$
is limit computable. Hence for any computable $G:\In\IN^\IN\to\IN^\IN$ the composition
$G\circ\J$ is limit computable by Proposition~\ref{prop:composition-limit}.

Let now $M_F$ be a limit Turing machine that computes $F:\In\IN^\IN\to\IN^\IN$.
We describe a Turing machine $M_G$ that computes a function $G:\In\IN^\IN\to\IN^\IN$
with $F(p)=G\circ\J(p)$ for all $p\in\dom(F)$.
There exists a computable function $r:\IN\to\IN$ such that the Turing machine with code $r\langle n,t\rangle$ halts 
upon input $p$ if and only if $M_F$ on input $p$ changes the output cell $n$ after more than $t$ steps.
Now $M_G$ on input $q$ works as follows: it simulates the machine $M_F$ on input $p=\J^{-1}(q)$, where
$\J^{-1}$ is the partial inverse of the Turing jump operator that is computable by Proposition~\ref{prop:jump-inversion},
and writes the output to some working tape. After simulating $M_F$ for $t$ steps at most the first $k$
cells on this working tape have been used for some $k\in\IN$. In this situation $M_G$ 
checks $qr\langle n,t\rangle$ for all $n=0,...,k$ and when the result is $0$ for 
some initial segment $n=0,...,j$, then $M_G$ copies the content of those cells $n=0,...,j$ to the output tape
that have not yet been written to the output.
This algorithm ensures that only those output cells which have already
stabilized are copied to the output. 
Thus, $M_G$ operates with a one-way output tape and computes a function $G:\In\IN^\IN\to\IN^\IN$ 
with $F(p)=G\circ\J(p)$ for all $p\in\dom(F)$. Since $\J$ is injective, one can restrict $G$ such that
one obtains $F=G\circ\J$.
\end{proof}

We obtain the following corollary, which is somehow dual to the statement of Corollary~\ref{cor:completeness-lim}.

\begin{corollary}
\label{cor:completeness-jump}
For any computable function $F:\In\IN^\IN\to\IN^\IN$ there exists a computable function
$G:\In\IN^\IN\to\IN^\IN$ such that $\J\circ F=G\circ\J$.
\end{corollary}

It is clear that $\J$ and $\lim$ cannot be swapped in Corollaries~\ref{cor:completeness-lim} and \ref{cor:completeness-jump}.
In particular, $\J$ is not transparent, since it has no computable values in its range.
However, it is not too difficult to see that $\J^{-1}$ is transparent.

It is a consequence of Corollary~\ref{cor:completeness-jump} that the Turing jump operator is monotone
with respect to Turing reducibility.

\begin{corollary}[Monotonicity of Turing jumps]
\label{cor:jump-monotone}
$q\leqT p\TO q'\leqT p'$ for all $p,q\in\IN^\IN$.
\end{corollary}

We also mention that the jump normal form theorem (Theorem~\ref{thm:jump}) yields as a non-uniform corollary a version of 
Shoenfield's limit lemma~\cite{Sho59}.

\begin{corollary}[Shoenfield's limit lemma 1959]
\label{cor:shoenfield}
A point $p\in\IN^\IN$ is limit computable if and only if $p\leqT0'$.
\end{corollary}

One can prove relativized forms of the jump normal form theorem (Theorem~\ref{thm:jump}) and
the limit normal form theorem (Theorem~\ref{thm:limit}). However, the exact relativization needs
some care.  
For every $q\in\IN^\IN$ we define the \emph{relativized jump operator} $\J_q:\IN^\IN\to\IN^\IN$
by $\J_q(p):=\J\langle q,p\rangle$. 
We say that $F:\In\IN^\IN\to\IN^\IN$ is \emph{limit computable relative to} $q\in\IN^\IN$ if
there exists a limit computable $G:\In\IN^\IN\to\IN^\IN$ such that $F(p)=G\langle q,p\rangle$ for all $p\in\dom(F)$.
Now the following theorem can be proved along the lines of the above results.

\begin{theorem}[Relativized limit computability]
\label{thm:relativized-limit-computability}
The following are equivalent for $q\in\IN^\IN$:
\begin{enumerate}
\item $F$ is limit computable relative to $q$,
\item $F=\lim\circ\,G$ for some $G$ that is computable relative to $q$,
\item $F=G\circ\J_q$ for some $G$ that is computable.
\end{enumerate}
\end{theorem}

The symmetry between the two normal forms does not fully extend to the relativized case since oracles act on the input side, and hence 
we need to use either functions that are computable relative to $q$ or the relativized jump $\J_q$. 
Related to this observation the equivalence expressed in the two normal form theorems is not fully uniform and, 
in fact, $\J^{-1}$ is not transparent in a topological sense (see~\cite[Proposition~9.13]{BBP12}).

It is rarely mentioned that $1$--generic points as they are used in computability theory \cite{Soa16}
can be characterized as points of continuity of the Turing jump $\J$. This was noticed in~\cite[Lemma~9.3]{BBP12}.

\begin{definition}[$1$--generic]
A point $p\in\IN^\IN$ is called \emph{$1$--generic} if and only if $\J$ is continuous at $p$.
\end{definition}

It is easy to see that $1$--generic points cannot be computable.
The following characterization of $1$--generic points is a consequence of Theorem~\ref{thm:jump}.
The $1$--generic points are exactly those at which every limit computable function is continuous.

\begin{corollary}[$1$--generic points]
\label{cor:1-generic}
A point $p\in\IN^\IN$ is $1$--generic if and only if every limit computable function $F:\In\IN^\IN\to\IN^\IN$
with $p\in\dom(F)$ is continuous at $p$.
\end{corollary}

While the ``only if'' direction is a consequence of Theorem~\ref{thm:jump}, one obtains the ``if'' direction
by applying the statement to $F=\J$.
We note that there are limit computable functions (such as the limit map $\lim$) that do not have
any $1$--generics in their domain, simply because they are not continuous at any point.

It is easy to see that restricted to $1$--generics the Turing jump operator is continuous and computable relative to the halting problem.

\begin{proposition}[Jump on $1$--generics]
\label{prop:jump-1-generic}
The Turing jump operator $\J:\IN^\IN\to\IN^\IN,p\mapsto p'$ restricted to the set 
of $1$--generics is computable relative to the halting problem.
\end{proposition}
\begin{proof}
The sets
\begin{enumerate}
\item $A:=\{(w,i)\in\IN^*\times\IN:$ the $i$--th Turing machine halts on all extensions of $w\}$,
\item $B:=\{(w,i)\in\IN^*\times\IN:$ the $i$--th Turing machine halts on no extension of $w\}$
\end{enumerate}
are both computable relative to $0'$. Given a $1$--generic $p\in\IN^\IN$ it is clear that for each $i\in\IN$ there
is a $w\prefix p$ such that $(w,i)\in A$ or $(w,i)\in B$. Depending on the answer we know that $\J(p)(i)=1$ or $\J(p)(i)=0$.
Hence, $\J$ restricted to $1$--generics is computable relative to the halting problem.
\end{proof}

As a consequence of the jump normal form (Theorem~\ref{thm:jump}) we obtain that every
limit computable function is computable relative to the halting problem, when restricted to the $1$--generic inputs.

\begin{corollary}[Limit computability on $1$--generics]
\label{cor:jump-1-generic}
Restricted to $1$--generics every limit computable $F:\In\IN^\IN\to\IN^\IN$ is computable relative to the halting problem.
\end{corollary}

As a non-uniform corollary of Proposition~\ref{prop:jump-1-generic}
we obtain that all $1$--generics are \emph{generalized low},
a property that is made more precise in the next well-known corollary~\cite[Lemma~2]{Joc77}.

\begin{corollary}[Jockusch 1977]
\label{cor:Jockusch}
$p'\equivT\langle p,0'\rangle$ for all $1$--generics $p\in\IN^\IN$.
\end{corollary}

Here the reduction $p'\leqT\langle p,0'\rangle$ follows directly from Proposition~\ref{prop:jump-1-generic},
and the inverse reduction even holds for arbitrary $p\in\IN^\IN$.

The points whose jump is below $0'$ also have a special name, they are called \emph{low}.

\begin{definition}[Low points]
\label{def:low}
A point $p\in\IN^\IN$ is called \emph{low}, if its Turing jump $p'$ is limit computable.
\end{definition}

It is clear that all computable points are low, and all limit computable $1$--generics are low by Corollary~\ref{cor:Jockusch}.
One can use this observation to show that the class of low points also contains non-computable points.
Using the \emph{low map} $\L:=\J^{-1}\circ\lim$  we obtain the following characterization of low points~\cite[Lemma~8.2]{BBP12}.

\begin{corollary}[Low points]
A point $p\in\IN^\IN$ is low if and only if there is a computable $q\in\IN^\IN$
such that $\L(q)=p$.
\end{corollary}

It follows from Corollaries~\ref{cor:jump-monotone} and \ref{cor:shoenfield}
that computable functions map low inputs to low outputs.
In \cite{BBP12} low functions were introduced that were further studied in \cite{BGM12,BHK17a}.

\begin{definition}[Low functions]
A function $F:\In\IN^\IN\to\IN^\IN$ is called \emph{low} if there is a computable $G:\In\IN^\IN\to\IN^\IN$
such that $F=\L\circ G$.
\end{definition}

This definition captures the idea that the result of $F$ is computed as a low point.
It is clear that low functions map computable inputs to low outputs. 
We obtain the following closure properties under composition.

\begin{proposition}[Composition with low functions]
\label{prop:composition-low}
Let $F,G:\In\IN^\IN\to\IN^\IN$ be functions.
If $F$ and $G$ are low then so is $G\circ F$. If $G$ is limit computable
and $F$ is low, then $G\circ F$ is limit computable.
\end{proposition}
\begin{proof}
Both observations are based on the jump normal form theorem (Theorem~\ref{thm:jump}) that
yields a computable $R:\In\IN^\IN\to\IN^\IN$ with $\lim=R\circ\J$. 
By Corollary~\ref{cor:completeness-lim} we also have a computable $S:\In\IN^\IN\to\IN^\IN$ with $R\circ\lim=\lim\circ S$.
This implies $\L\circ\L=\J^{-1}\circ\lim\circ\J^{-1}\circ\lim=\J^{-1}\circ R\circ\J\circ\J^{-1}\circ\lim=\L\circ S$
and $\lim\circ\L=R\circ\J\circ\J^{-1}\circ\lim=R\circ\lim$. The cases of the composition of general low and limit computable functions
can easily be derived from these observations.
\end{proof}

In fact, the low functions form the largest class of functions that can be composed with limit computable functions from the left
without leaving the class of limit computable functions  (see also \cite[Proposition~14.16]{BHK17a} for a related result).

\begin{corollary}[Low functions]
\label{cor:low-functions}
A function $F:\In\IN^\IN\to\IN^\IN$ is low if and only if $G\circ F$ is limit computable for every limit computable $G:\In\IN^\IN\to\IN^\IN$.
\end{corollary}

The ``only if'' direction is a direct consequence of Proposition~\ref{prop:composition-low}.
The ``if'' direction follows from the limit normal form theorem if one applies the assumption to $G=\J$.
As a consequence of Proposition~\ref{prop:composition-finite} we get the following corollary.

\begin{corollary}[Finite mind changes are low]
\label{cor:finite-low}
Every function $F:\In\IN^\IN\to\IN^\IN$ that is computable with finitely many mind changes is low.
\end{corollary}

Hence, in contrast to low points $p\in\IN^\IN$, we have a lot of natural examples of low functions,
such as the equality test $\E$.

We close this section by mentioning briefly that there is a characterization of functions that are computable
with finitely many mind changes that is analogous to Theorem~\ref{thm:limit}.

\begin{theorem}[Discrete limit normal form]
\label{thm:discrete-limit}
A function $G:\In\IN^\IN\to\IN^\IN$ is computable with finitely many mind changes if and only
if there exists a computable function $F:\In\IN^\IN\to\IN^\IN$, such that
$G=\lim\nolimits_\Delta\circ\; F$.
\end{theorem}

We leave the simple proof to the reader (see \cite{BBP12}).

\section{Computability Relative to the Halting Problem}
\label{sec:computability-halting-problem}

Now we consider functions that are computable relative to the halting problem $0'$. 
It turns out that they play a dual r{\^o}le to the low functions that are characterized in terms of $\L=\J^{-1}\circ\lim$.
The functions that are computable relative to the halting problem can be characterized in terms of $\HP:=\J\circ\lim^{-1}:\IN^\IN\mto\IN^\IN$. 

The first result shows that $\HP$ is computable relative to the halting problem. 
In order to prove this we need to invert limits such that we obtain control on the jump of the resulting converging sequence.
Roughly speaking, this is possible since there are many sequences converging to a given point, and this
gives us enough freedom to choose a suitable one. We recall that the \emph{composition} of 
two multi-valued functions $f:\In X\mto Y$ and $g:\In Y\mto Z$ is defined by $(g\circ f)(x):=\{z\in Z:(\exists y\in f(x))\;z\in g(y)\}$
with $\dom(g\circ f):=\{x\in X:x\in\dom(f)$ and $f(x)\In\dom(g)\}$.
A multi-valued $f:\In\IN^\IN\mto\IN^\IN$ is called \emph{computable (relative to some oracle $q$)} if there is
a single-valued $F:\In\IN^\IN\to\IN^\IN$ with the corresponding property and such that 
$F(p)\In f(p)$ for all $p\in\dom(f)$ and $\dom(f)\In\dom(F)$.
We recall that a function $F:\In\IN^\IN\to\IN^\IN$ is called \emph{computable relative to some oracle $q\in\IN^\IN$} if there is some computable function $G:\In\IN^\IN\to\IN^\IN$ such that
$F(p)=G\langle q,p\rangle$ for all $p\in\dom(F)$. 
In terms of computability theory, the next proof uses the \emph{finite extension method}.

\begin{theorem}[Limit inversion]
\label{thm:limit-inversion}
$\HP=\J\circ\lim^{-1}$ is computable relative to the halting problem.
\end{theorem}
\begin{proof}
We need to show that there exists a function $I:\IN^\IN\to\IN^\IN$ with $\lim\circ I=\id$ and such that
$\J\circ I$ is computable relative to $0'$.
Given $p\in\IN^\IN$ we describe the computation of $I$ relative to $0'$ by an inductive construction
in $i\in\IN$.
We start with $w_0:=\varepsilon\in\IN^*$ and $t_0:=0\in\IN$.
Given a finite sequence of words $w_0,...,w_{t_i}\in\IN^*$
together with numbers $t_0<t_1<...<t_i$ and bits $b_0,...,b_{i-1}\in\{0,1\}$
we describe how we determine $t_{i+1}\in\IN$, $w_{t_{\underline{i}}+1},...,w_{t_{\underline{i+1}}}\in\IN^*$ and $b_i\in\{0,1\}$ in stage $i$ of the construction.\footnote{For notational clarity, we have underlined the indexes of $t$;
the line has no mathematical meaning.}
We let $r_0:=p|_0w_0\widehat{0}=\widehat{0}$ and $r_{n+1}:=(p|_{n+1}w_{t_{\underline{n}}+1}\widehat{0},...,p|_{n+1}w_{t_{\underline{n+1}}}\widehat{0})$ for all $n<i$.
Now we consider the question whether the $i$--th Turing machine halts on
\[\langle r_0,...,r_i,p|_{i+1}q_{t_i+1},p|_{i+1}q_{t_i+2},p|_{i+1}q_{t_i+3},...\rangle\]
for some $q_{t_i+n}\in\IN^\IN$ with $n\in\IN$.
Whether or not this is the case can be decided with the help of $0'$ since the answer depends
only on the finite portion $p|_{i+1}$ of $p$ read so far.
If the outcome of the decision is positive, then already a finite number $q_{t_{\underline{i}}+1},...,q_{t_{\underline{i+1}}}\in\IN^\IN$ for some $t_{i+1}>t_i$ and, in fact, finite prefixes of these
are sufficient for the machine $i$ to halt. In this case we can compute such a $t_{i+1}$ and corresponding words $w_{t_i+n}\in\IN^*$ such that $q_j:=w_j\widehat{0}$ for $j\in\{t_i+1,...,t_{i+1}\}$
satisfy the condition, and we set $b_i:=1$. Otherwise, if the outcome of the decision is negative, then we
choose $w_{t_i+1}:=\varepsilon$, $t_{i+1}:=t_i+1$ and $b_i:=0$. 
If we continue inductively in this way, 
then we can compute $I(p):=\langle r_0,r_1,r_2,...\rangle$ with the help of $0'$ uniformly in $p$.
Since we have used longer and longer prefixes $p|_n$ of $p$ in the construction of the $r_n$, 
the value $I(p)$ satisfies $\lim\circ I(p)=p$.
Moreover, the construction guarantees that we can compute $\J\circ I$ with the help of $0'$,
namely $\J\circ I(p)=(b_0,b_1,b_2,...)$.
\end{proof}

We note that $\HP$ is, in particular, continuous, and $\HP^{-1}=\lim\circ\J^{-1}$ is computable by Theorem~\ref{thm:jump}.
Theorem~\ref{thm:limit-inversion} yields the following characterization of functions computable relative to the 
halting problem. The interesting point of the following theorem is not that there is an object $\HP$ that characterizes
functions that are computable relative to the halting problem (the problem $\langle\id\times\chi_{0'}\rangle$ with the characteristic function $\chi_{0'}$ of the halting problem
would be a simpler such example), but the point is that our particular $\HP$ has this property.

\begin{theorem}[Computability relative to the halting problem]
\label{thm:computable-halting-problem}
A function $F:\In\IN^\IN\to\IN^\IN$ is computable relative to the halting problem if and 
only if there is a computable function $G:\In\IN^\IN\to\IN^\IN$ such that $F=G\circ\HP$.
\end{theorem}
\begin{proof}
By Theorem~\ref{thm:limit-inversion} $\HP$ is computable relative to $0'$. Hence $G\circ\HP$ is 
computable relative to the halting problem for every computable $G:\In\IN^\IN\to\IN^\IN$.
Let now $F$ be computable relative to the halting problem $0'$. Then there is a computable
$S:\In\IN^\IN\to\IN^\IN$, such that $F(p)=S\langle 0',p\rangle$ for all $p\in\dom(F)$.
By Theorem~\ref{thm:jump} we know that there is a computable function $T:\In\IN^\IN\to\IN^\IN$ such that
$T\circ\J=\lim$, i.e., $T\circ\HP=\id$. On the other hand, there is a computable function
$r:\IN\to\IN$ such that Turing machine $r(i)$ halts on all inputs if and only if machine $i$
halts on input $\widehat{0}$. Hence $\J(\widehat{0})(i)=\J(p)(r(i))$ for all $p\in\IN^\IN$ and $i\in\IN$.
Then $R:\IN^\IN\to\IN^\IN$ with $R(q)(i):=qr(i)$ is computable and so is $G:=S\circ\langle R,T\rangle$.
We obtain
\[F(p)=S\langle\J(\widehat{0}),p\rangle=S\langle R\circ\HP(p),T\circ\HP(p)\rangle=G\circ\HP(p)\] 
for all $p\in\dom(F)$, hence $F=G\circ\HP$ if $G$ is restricted suitably (which is possible
since $\HP(p)\cap\HP(q)=\emptyset$ for $p\not=q$.)
\end{proof}

We note that this result implies that $\HP$ is not computable (hence it is an interesting example
of a natural problem that is continuous and not computable).
Another interesting consequence of our results is
that the functions that are computable relative to the halting problem form the largest class of functions
that, when composed with a limit computable function from the right,  yield a limit computable function.

\begin{corollary}[Computability relative to the halting problem]
\label{cor:computable-halting-problem}
A function $F:\In\IN^\IN\to\IN^\IN$ is computable relative to the halting problem if
and only if $F\circ G$ is limit computable for every limit computable $G:\In\IN^\IN\to\IN^\IN$.
\end{corollary}

The ``if'' direction follows from Theorems~\ref{thm:computable-halting-problem} and \ref{thm:jump}, when one chooses $G=\lim$.
The ``only if'' direction is a consequence of Theorems~\ref{thm:computable-halting-problem} and \ref{thm:limit}.
A comparison of Corollaries~\ref{cor:computable-halting-problem} and \ref{cor:low-functions} shows
that the notions of a low function is dual to the notion of a function that is computable relative to the halting
problem. Both classes should be naturally studied alongside the limit computable functions.

As another non-uniform side result of the limit inversion theorem (Theorem~\ref{thm:limit-inversion}) we obtain a classical result
from computability theory, the Friedberg jump inversion theorem~\cite{Fri57}.

\begin{corollary}[Friedberg jump inversion theorem 1957]
\label{cor:friedberg-jump-inversion}
For every $q\in\IN^\IN$ with $0'\leqT q$ there exists a $p\in\IN^\IN$
with $p'\equivT q$.
\end{corollary}
\begin{proof}
Given a fixed $q\in\IN^\IN$ with $0'\leqT q$ we can compute $p:=I(q)$ and
$p'=\J\circ I(q)$ with $I$ from the proof of Theorem~\ref{thm:limit-inversion}. In particular, $p'\leqT q$.
On the other hand, by the same theorem and  Theorem~\ref{thm:jump} we obtain $q=\lim(p)\leqT p'$.
\end{proof}

As another corollary we obtain the announced extension of Corollary~\ref{cor:completeness-lim}.

\begin{corollary}
\label{cor:computability-halting-problem}
For every function $F:\In\IN^\IN\to\IN^\IN$ that is computable relative to the halting problem there exists a computable function
$G:\In\IN^\IN\to\IN^\IN$ such that $\lim\circ\, G=F\circ\lim$.
\end{corollary}

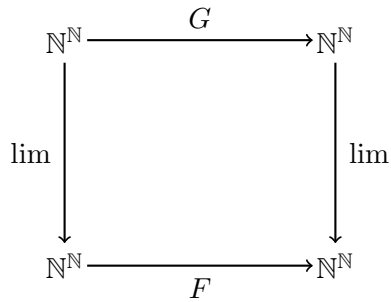
\begin{figure}[htb]
\begin{center}
\begin{tikzpicture}[scale=0.6]

\node at (-3,5) {$\IN^\IN$};
\node at (3,5) {$\IN^\IN$};
\node at (-3,0) {$\IN^\IN$};
\node at (3,0) {$\IN^\IN$};
\node at (0,5.5) {$G$};
\node at (0,-0.5) {$F$};
\node at (-3.75,2.5) {$\lim$};
\node at (3.75,2.5) {$\lim$};
\node (v1) at (-2.5,4.5) {};
\node (v2) at (2.5,0.5) {};

\draw [->,thick](-2.5,5) -- (2.5,5);
\draw [->,thick](-2.5,0) -- (2.5,0);
\draw [->,thick](-3,4.5) -- (-3,0.5);
\draw [->,thick](3,4.5) -- (3,0.5);
\end{tikzpicture}
\end{center}
\caption{Limit diagram.}
\label{fig:limit-control}
\end{figure}

This result can be interpreted as a statement about the diagram in Figure~\ref{fig:limit-control}.
Vice versa, if $G$ is an arbitrary computable function that transfers converging sequences into 
converging sequences (not necessarily in an extensional way), then we can mimic the behavior
of $G$ on the limits by a function that is computable relative to the halting problem.
This follows from Theorems~\ref{thm:computable-halting-problem} and \ref{thm:jump}
and was proved directly with a somewhat more involved proof that does not exploit
the Galois connection between Turing jumps and limits in \cite[Theorem~14.11]{BHK17a}.

\begin{corollary}[B., Hendtlass and Kreuzer 2017]
\label{cor:limit-control}
For all computable $G:\In\IN^\IN\to\IN^\IN$ the multi-valued function $\lim\circ G\circ\lim^{-1}$
is computable relative to the halting problem.
\end{corollary}

We close this section with a uniform version of Corollary~\ref{cor:limit-control}.
For the formulation we need a total representations $\Phi$ of all continuous functions $F:\In\IN^\IN\to\IN^\IN$
with computable $G_\delta$--domain that satisfies a utm- and smn-theorem~\cite[Theorem~2.3.5]{Wei00}.
The result is tailor-made to yield a proof of Theorem~\ref{thm:jumps-products-exponentials}~(3).

\begin{theorem}[Uniform limit control theorem]
\label{thm:limit-control}
There exists a computable $R:\IN^\IN\to\IN^\IN$ such that
$\Phi_{\lim\circ R(q)}(p)\in\lim\circ\Phi_q\circ\lim^{-1}(p)$ for all $p,q\in\IN^\IN$ with
$p\in\dom(\lim\circ\Phi_q\circ\lim^{-1})$.
\end{theorem}
\begin{proof}
By a relativized version of Theorem~\ref{thm:limit-inversion}
there exist computable $I,S:\In\IN^\IN\to\IN^\IN$ such that $\lim\circ I_{q'}=\id$
and $S_{q'}=\J_q\circ I_{q'}$ for every $q\in\IN^\IN$.
This can be proved exactly as in the proof of Theorem~\ref{thm:limit-inversion} except that we check whether the $i$--th Turing machine
halts upon input of
\[\langle q,\langle r_0,...,r_i,p|_{i+1}q_{t_i+1},p|_{i+1}q_{t_i+2},p|_{i+1}q_{t_i+3},...\rangle\rangle,\]
which can be decided using $q'$.
Now, by applying the jump normal form theorem (Theorem~\ref{thm:jump}) to $\Phi$ we obtain
a computable function $H$ with $\lim\circ\Phi_q=H\circ\J_q$. 
Then $H\circ S$ is computable, and by the smn-theorem for $\Phi$
there exists a total computable function $P$ such that $\Phi_{P(r)}(p)=H\circ S\langle r,p\rangle$ for all $r,p\in\IN^\IN$.
Finally, by the limit normal form theorem (Theorem~\ref{thm:limit}) there exists a computable function $R$ such that
$\lim\circ R=P\circ\J$. Such an $R$ is necessarily total. Altogether, we obtain
\[\Phi_{\lim R(q)}(p)=\Phi_{P\J(q)}(p)=HS\langle q',p\rangle=H\J_{q}I_{q'}(p)=\lim \Phi_q I_{q'}(p)\in\lim\circ\Phi_q(\lim\nolimits^{-1}(p))\]
for all $p,q\in\IN^\IN$ with $\lim^{-1}(p)\In\dom(\lim\circ\Phi_q)$.
\end{proof}

We note that the basic problems $\lim,\J^{-1}$ are generators of a monoid which can be used to characterize
all sorts of other computability theoretic properties. For instance $\J^{-1}\J^{-1}\lim\lim$
characterizes the property of being \emph{low$_2$} and $\J^{-1}\lim\lim$ the property of being
\emph{low relative to the halting problem} (see also \cite{BGM12}).

\section{Limit computability on represented spaces}
\label{sec:represented-spaces}

The purpose of this section is to transfer our results on limit computable maps to represented spaces.
We recall that a \emph{representation} of a set $X$ is a surjective
function $\delta:\In\IN^\IN\to X$.
In this situation we say that $(X,\delta)$ is a 
\emph{represented space}.
Any point $p\in\IN^\IN$ with $\delta(p)=x$ is called 
a \emph{name} of $x$, and the word ``representation'' is 
reserved for the map $\delta$ as such.
A \emph{problem} is a multi-valued partial function
$f:\In X\mto Y$ on represented spaces $X$ and $Y$.
If $(X,\delta_X)$ and $(Y,\delta_Y)$ are represented spaces
and $f:\In X\mto Y$ is a problem, then
$F:\In\IN^\IN\to\IN^\IN$ \emph{realizes}
$f$, in symbols $F\vdash f$, if $\delta_YF(p)\in f\delta_X(p)$
holds for all $p\in\dom(f\delta_X)$.

\begin{definition}[Computable functions]
\label{def:computable-realizer}
A problem $f$ is called \emph{continuous}, \emph{computable}, \emph{limit computable}, \emph{low} or \emph{computable with respect to the halting problem},
if it has a realizer with the corresponding property.
\end{definition}

Likewise, other properties can be transferred from realizers to problems.
By $\CC(X,Y)$ we denote the  \emph{set of total continuous functions} $f:X\to Y$.
It is clear that all properties related to closure under composition that have been discussed in the previous section can 
be transferred to problems. This is because $F\vdash f$ and $G\vdash g$ implies $F\circ G\vdash f\circ g$.
We can also transfer properties of points to represented spaces using representations.

\begin{definition}[Computable points]
A point $x\in X$ in a represented space $X$ is called \emph{computable}, \emph{limit computable} or \emph{low} 
if it has a name with the corresponding property.
\end{definition}

We note that the notion of $1$--genericity is an example of a property that should not be defined
via names since it is not invariant under equivalent representations (see Proposition~\ref{prop:non-generic-names} and the discussion afterwards).
 
Our main goal here is to transfer the limit normal form theorem and the
jump normal form theorem (Theorems~\ref{thm:limit} and \ref{thm:jump}) to problems on represented spaces.
Since we can incorporate limits and jumps into the represented spaces, these
normal forms can be expressed very neatly. For this purpose we need the following
concepts of jumps on represented spaces. 

\begin{definition}[Jumps]
Let $(X,\delta_X)$ be a represented space and $\delta:=\delta_X$, let $T:\In\IN^\IN\to\IN^\IN$ be surjective, and
let $S:\In\IN^\IN\mto\IN^\IN$ be such that $S^{-1}:\In\IN^\IN\to\IN^\IN$ is single-valued and surjective.
 We define the following representations of $X$:
\begin{enumerate}
\item $\delta^T:=\delta_{X^T}:=\delta_X\circ T$ \hfill (\emph{$T$--jump})
\item $\delta_S:=\delta_{X_S}:=\delta_X\circ S^{-1}$ \hfill (\emph{$S^{-1}$--jump})
\end{enumerate}
We denote the represented spaces $(X,\delta_{X^T})$ and $(X,\delta_{X_S})$ for short by $X^T$ and $X_S$, respectively.
In the special case of $T=\lim$ and $T=\lim_\Delta$, we define 
\begin{enumerate}
\item[(3)] $\delta':=\delta_{X'}:=\delta_X\circ\lim$ \hfill (jump)
\item[(4)] $\delta^\Delta:=\delta_{X^\Delta}:=\delta_X\circ\lim_\Delta$ \hfill (discrete jump)
\end{enumerate}
We denote the represented spaces $(X,\delta_{X'})$ and $(X,\delta_{X^\Delta})$ also by $X'$ and $X^\Delta$, respectively.
\end{definition}

We will apply this concept of a jump in the case of $T=\lim$, $T=\L$, $S=\HP$ and $S=\J$.
The special jumps for $T=\lim$ and $T=\lim_\Delta$ (the limit on Baire space with respect to the discrete metric)
were originally defined by Ziegler~\cite{Zie07} and later also studied by the author, de Brecht and Pauly~\cite{BBP12}.
More general concepts of jumps were also studied in the context of effective descriptive set theory by de Brecht~\cite{dBre14} and de Brecht and Pauly~\cite{PdB13,PdB14,PdB15}.

A common feature of the maps $\lim,\lim_\Delta,\J^{-1},\L$ and $\HP^{-1}$ is that they are surjective and transparent.
It was noted by de Brecht~\cite{dBre14} that the functor $X\mapsto X^T$ for surjective and transparent $T$ can be seen
as an endofunctor on the class of represented spaces (that leaves the maps unchanged). 
This is made precise by the following result.

\begin{proposition}[The jump as endofunctor]
\label{prop:jump-endofunctor}
If $f:\In X\mto Y$ is a computable problem and $T:\In\IN^\IN\to\IN^\IN$ is transparent and surjective
then it follows that $f$ considered as a problem of type
$f:\In X^T\mto Y^T$  is computable too.
\end{proposition}
\begin{proof}
 If $f:\In X\mto Y$ is computable, then it has a computable realizer $F:\In\IN^\IN\to\IN^\IN$. Since $T$ is transparent,
there is a computable $G:\In\IN^\IN\to\IN^\IN$ with $F\circ T=T\circ G$. Hence $G$ is a computable realizer of $f:\In X^T\mto Y^T$.
\end{proof}

Through some results presented in this section we will get a better understanding of the
types of problems mentioned in Proposition~\ref{prop:jump-endofunctor} for the specific maps $T$ that we are interested in.

We start with a result that characterizes limit computable problems
and shows that the jump $Y'$ on the output side can be balanced by $X_\J$
on the input side. 

\begin{theorem}[Limit computability]
\label{thm:limit-computability}
Let $f:\In X\mto Y$ be a problem. Then the following are equivalent:
\begin{enumerate}
\item $f:\In X\mto Y$ is limit computable,
\item $f:\In X\mto Y'$ is computable,
\item $f:\In X_\J\mto Y$ is computable.
\end{enumerate}
\end{theorem}
\begin{proof}
The equivalences are consequences of the limit normal form theorem (Theorem~\ref{thm:limit})
and the jump normal form theorem (Theorem~\ref{thm:jump}).
\end{proof}

We note that this result can be interpreted such that jumps $X'$ and $X_\J$ are adjoint to each other.
However, this is a pure computability theoretic adjointness that does not relativize to a topological adjointness.
The proof of the direction (3)$\TO$(2) can be extended to a uniform proof, 
but this is not so for the direction (2)$\TO$(3).
In fact, $\CC((\IN^\IN)_\J,\IN^\IN)\subsetneqq\CC(\IN^\IN,{(\IN^\IN)}')$, since $\J^{-1}$ is not topologically transparent.
However, we can derive the following relativized version of Theorem~\ref{thm:limit-computability} from
Theorem~\ref{thm:relativized-limit-computability}.

\begin{corollary}[Relativized limit computability]
\label{cor:limit-computability}
Let $f:\In X\mto Y$ be a problem and $q\in\IN^\IN$. Then the following are equivalent:
\begin{enumerate}
\item $f:\In X\mto Y$ is limit computable relative to $q$,
\item $f:\In X\mto Y'$ is computable relative to $q$,
\item $f:\In X_{\J_q}\mto Y$ is computable.
\end{enumerate}
\end{corollary}

We can transfer our characterization of low problems. 

\begin{theorem}[Low computability]
\label{thm:low}
Let $f:\In X\mto Y$ be a problem. Then the following are equivalent:
\begin{enumerate}
\item $f:\In X\mto Y$ is low,
\item $f:\In X\mto Y^\L$ is computable,
\item $f:\In X^\L\mto Y^\L$ is computable,
\item $f:\In X_\J\mto Y_\J$ is computable,
\item $f:\In X_\J\mto Y_\J$ is computable relative to the halting problem $0'$.
\end{enumerate}
\end{theorem}
\begin{proof}
The equivalence of (1) and (2) is a direct consequence of the definition of low maps on Baire space.
Since low maps are closed under composition by Proposition~\ref{prop:composition-low}, it follows that (2) implies (3).
Since $\id:X\to X^\L$ is computable, it follows that (3) implies (2). 
By Theorem~\ref{thm:limit-computability} (2) and (4) are equivalent  since $X^\L=(X_\J)'$.
By the statement ``(1)$\iff$(4)'' of Theorem~\ref{thm:computability-halting-problem} (which independently
follows from Corollaries~\ref{cor:computability-halting-problem} and \ref{cor:limit-control}) we obtain that (3) and (5) are equivalent.
\end{proof}

Finally, we obtain a dual characterization of the problems that are computable relative to the halting problem.\footnote{We note that the equivalence of (1) and (4) shows that \cite[Theorem~23~(c)]{Zie07a} is not correct.}

\begin{theorem}[Computability relative to the halting problem]
\label{thm:computability-halting-problem}
Let $f:\In X\mto Y$ be a problem. Then the following are equivalent:
\begin{enumerate}
\item $f:\In X\mto Y$ is computable relative to the halting problem $0'$,
\item $f:\In X_\HP\mto Y$ is computable,
\item $f:\In X_\HP\mto Y_\HP$ is computable,
\item $f:\In X'\mto Y'$ is computable,
\item $f:\In X'\mto Y'$ is low.
\end{enumerate}
\end{theorem}
\begin{proof}
The equivalence of (1) and (2) is a consequence of Theorem~\ref{thm:computable-halting-problem}.
Since functions computable relative to $0'$ are closed under composition,
it follows again by Theorem~\ref{thm:computable-halting-problem} that (2) implies (3).
Since $\id:Y_\HP\to Y$ is computable, it is clear that (3) implies (2).
The equivalence of (2) and (4) follows from Theorem~\ref{thm:limit-computability}
since $X_\HP=(X')_{\J}$.
The equivalence of (3) and (5) follows from Theorem~\ref{thm:low}.
\end{proof}

The equivalence of (1) and (4) can also be directly derived from Corollaries~\ref{cor:computability-halting-problem} and \ref{cor:limit-control}.
Likewise we obtain the following relativized version.

\begin{theorem}[Computability relative to an oracle]
\label{thm:computability-oracle}
Let $f:\In X\mto Y$ be a problem and $q\in\IN^\IN$. Then the following are equivalent:
\begin{enumerate}
\item $f:\In X\mto Y$ is computable relative to $q'$,
\item $f:\In X'\mto Y'$ is computable relative to $q$.
\end{enumerate}
\end{theorem}
\begin{proof}
The implication from (2) to (1) is a consequence of Theorem~\ref{thm:limit-control}.
For the implication from (1) to (2) we need a relativized version of Corollary~\ref{cor:computability-halting-problem}
that can actually be derived from this corollary. Let $\U\langle q,p\rangle:=\Phi_q(p)$ be the universal computable function.
Then by Corollary~\ref{cor:computability-halting-problem} there exists a computable $G$ such that $\lim\circ G=\U\circ\lim$.
There is also a computable $R:\IN^\IN\to\IN^\IN$ with $\lim\circ R\langle q,p\rangle=\langle\lim q,\lim p\rangle$.
If we denote by $R_q$ the function with $R_q(p):=R\langle q,p\rangle$ then we obtain $\lim\circ G\circ R_q=\Phi_{\lim q}\circ\lim$.
Hence, $G\circ R_q$ is a realizer for $f:\In X'\mto Y'$
if $\Phi_{\lim q}$ is a realizer for $f:\In X\mto Y$. This proves the claim.
\end{proof}

We immediately obtain the following corollary, where $q^{(n)}$ denotes the $n$--th Turing jump of $q\in\IN^\IN$
and $X^{(n)}$ denotes the $n$--th jump of the represented space $X$.

\begin{corollary}[Computability relative to higher jumps]
\label{cor:higher-jumps}
Let $f:\In X\mto Y$ be a problem and $n\in\IN$. Then the following are equivalent:
\begin{enumerate}
\item $f:\In X\mto Y$ is computable relative to $0^{(n)}$,
\item $f:\In X^{(n)}\mto Y^{(n)}$ is computable.
\end{enumerate}
\end{corollary}

Other characterizations can be derived as conclusions of the results provided in this section.
Now we want to show that jumps interact nicely with products and function space constructions.
We recall that for two represented spaces $(X,\delta_X)$ and $(Y,\delta_Y)$ we can define a representation
of $X\times Y$, of $X^\IN$ and of $\CC(X,Y)$ as follows:
\begin{enumerate}
\item $\delta_{X\times Y}:\In\IN^\IN\to X\times Y,\langle p,q\rangle\mapsto(\delta_X(p),\delta_Y(q))$.
\item $\delta_{X^\IN}:\In\IN^\IN\to X^\IN,\langle p_0,p_1,p_2,...\rangle\mapsto(\delta_X(p_n))_{n\in\IN}$. 
\item $\delta_{\CC(X,Y)}:\In\IN^\IN\to\CC(X,Y)$ by $\delta_{\CC(X,Y)}(p)=f:\iff \Phi_p\vdash f$. 
\end{enumerate}
It is known that these representations make evaluation and currying computable~\cite[Lemmas~3.3.14, 3.3.16 and Theorem~3.3.15]{Wei00} and that sequences
can be identified with continuous functions.

\begin{facts}[Function space]
\label{fact:function-space}
The following are computable for all represented spaces $X,Y$:
\begin{enumerate}
\item $\ev:\CC(X,Y)\times X\to Y,(f,x)\mapsto f(x)$,
\item $\mathrm{cur}:\CC(X\times Y,Z)\to\CC(X,\CC(Y,Z)),f\mapsto(x\mapsto (y\mapsto f(x,y)))$,
\item $\id:\CC(\IN,X)\to X^\IN$ and its inverse.
\end{enumerate}
\end{facts}

We obtain the following result that shows how jumps interact with products and function space constructions.
We call a problem $f:X\to Y$ a \emph{computable isomorphism} if $f$ is computable and bijective and its
inverse $f^{-1}$ is computable too.

\begin{theorem}[Jumps with products and exponentials]
\label{thm:jumps-products-exponentials}
Let $X$ and $Y$ be represented spaces. Then the following are computable isomorphisms:
\begin{enumerate}
\item $\id:(X\times Y)'\to X'\times Y'$,
\item $\id:(X^\IN)'\to(X')^\IN$,
\item $\id:\CC(X,Y)'\to\CC(X',Y')$.
\end{enumerate}
In particular, $\CC(X',Y')$ is exactly the set of continuous functions $f:X\to Y$.
\end{theorem}
\begin{proof}
The first two statements follow from the fact that the tupling functions are continuous:
\[\lim_{i\to\infty}\langle p_i,q_i\rangle=\langle\lim_{i\to\infty}p_i,\lim_{i\to\infty}q_i\rangle
\mbox{ and }\lim_{i\to\infty}\langle p_{0i},p_{1i},...\rangle=\langle\lim_{i\to\infty}p_{0i},\lim_{i\to\infty}p_{1i},...\rangle.\]
Now we consider the evaluation map $\ev:\CC(X,Y)\times X\to Y$, which is computable.
By Proposition~\ref{prop:jump-endofunctor} it follows that also $\ev:(\CC(X,Y)\times X)'\to Y'$ is computable, and
hence by (1) we obtain that $\ev:\CC(X,Y)'\times X'\to Y'$ is computable.
Since currying is computable, it follows that
$\id:\CC(X,Y)'\to\CC(X',Y')$ is computable.
We still need to prove that $\id:\CC(X',Y')\to\CC(X,Y)'$ is computable.
But this follows from the uniform version of the uniform limit control theorem (Theorem~\ref{thm:limit-control}).
\end{proof}

We should warn the reader that we have an ambiguity in the terminology
that one has to keep in mind and that is expressed in the following corollary.

\begin{corollary}[Limit computable points in function spaces]
\label{cor:limit-function-points}
Let $X,Y$ be represented spaces and $f\in\CC(X,Y)$.
Then $f$ is limit computable as a point in $\CC(X,Y)$ if and only if 
$f:X\to Y$ is computable relative to the halting problem as a function.
\end{corollary}

However, one can also see this as fit in terminology since for points in Baire space being limit computability
is equivalent to being computable relative to the halting problem by Shoenfield's limit lemma (Corollary~\ref{cor:shoenfield}).
Hence $f$ being computable relative to the halting problem as a point in $\CC(X,Y)$ is equivalent
to $f:X\to Y$ being computable relative to the halting problem as a function.
The equivalence between computability relative to the halting problem and limit computability for points 
extends to sequences. Theorem~\ref{thm:jumps-products-exponentials} together with the fact
that $\id:\CC(\IN,X)\to X^\IN$ is a computable isomorphism imply the following.

\begin{corollary}[Limit computable sequences]
\label{cor:sequences}
The identity $\id:\CC(\IN,X)'\to\CC(\IN,X')$ is a computable isomorphism.
In particular, the limit computable functions $f:\IN\to X$ are exactly the
functions $f:\IN\to X$ that are computable relative to the halting problem.
\end{corollary}

The jump $X\mapsto X_\J$ does not preserve products and function spaces
in the way $X\mapsto X'$ does according to Theorem~\ref{thm:jumps-products-exponentials}.
Basically, all our negative results in this direction can be derived from Spector's jump inversion theorem (see \cite{Spe56} or \cite[Proposition~V.2.26]{Odi89}).  
In particular, it implies that the notion of a low function is incomparable with the notion of a function
that is computable relative to a low oracle.\footnote{A referee provided an interesting alternative proof of this 
result, using Chaitin's $\Omega=\langle \Omega_0,\Omega_1\rangle$. The even and odd parts $\Omega_0$ and $\Omega_1$, respectively,
are low according to \cite[Theorem~15.2.3]{DH10} and hence $F:p\mapsto\langle\Omega_0,p\rangle$ and 
$f:i\mapsto\Omega_i$ are alternative concrete examples that satisfy Proposition~\ref{prop:low-relative-low}.}

\begin{proposition}[Low functions]
\label{prop:low-relative-low}
There is a Lipschitz continuous function $F:\IN^\IN\to\IN^\IN$ that is computable relative to a low oracle,
but that is not low as a function. There is also a low function $f:\{0,1\}\to\IN^\IN$ that is not computable
relative to a low oracle.
\end{proposition}
\begin{proof}
By Spector's jump inversion theorem there are low $p,q\in\IN^\IN$ such that $\langle p,q\rangle$ is not low.
Hence the function
$F:\IN^\IN\to\IN^\IN,r\mapsto\langle q,r\rangle$ 
does not map the low $p\in\IN^\IN$ to a low value $F(p)=\langle q,p\rangle$. 
That is, $F$ is not low, but computable relative to the low point $q$.
It is easy to see that $F$ is Lipschitz continuous with Lipschitz constant $1$.

There are also computable $r,s\in\IN^\IN$ such that $\L(r)=p$ and $\L(s)=q$.
We consider the functions $f,g:\{0,1\}\to\IN^\IN$ defined by
\[f(i):=\left\{\begin{array}{ll} p & \mbox{if $i=0$}\\ q & \mbox{otherwise}\end{array}\right.\mbox{ and }
   g(i):=\left\{\begin{array}{ll} r & \mbox{if $i=0$}\\ s & \mbox{otherwise}\end{array}\right..\]
Then $g$ is computable and $f=\L\circ g$. Hence $f$ is low. Let us assume that $f$ is computable with
respect to a low oracle $t\in\IN^\IN$. Then $p=f(0)\leqT t$ and $q=f(1)\leqT t$ follows and hence
$\langle p,q\rangle\leqT t$, which is a contradiction since $\langle p,q\rangle$ is not low.
\end{proof}

We will see several further
low functions that are not computable relative to a low oracle in section~\ref{sec:limit-halting}.

As a preparation for the following results we  show that relative computability can also
be characterized via function spaces if the corresponding oracle class
is closed downwards by Turing reducibility.

\begin{lemma}[Relative computability]
Let $F:\In\IN^\IN\to\IN^\IN$ be a function and $q\in\IN^\IN$. Then the following are equivalent:
\begin{enumerate}
\item $F$ is computable relative to $q$,
\item $(\exists r\leqT q)(\forall p\in\dom(F))\; F(p)=\Phi_r(p)$.
\end{enumerate}
\end{lemma}
\begin{proof}
(1) implies (2) by the smn-theorem and (2) implies (1) by the utm-theorem.
\end{proof}

We can apply this in particular to the class of functions that are computable with respect
to a low oracle.

\begin{corollary}[Computability relative to a low oracle]
\label{cor:relative-low}
Let $X,Y$ be represented spaces. Then
$\CC(X,Y)^\L$ is the class of functions that are computable relative to a low oracle.
\end{corollary}

Now we get the following negative result on how $\L$ and $\J$ behave on function spaces.
We note that the given identities are not surjective,
since the spaces $\CC(X^\L,Y^\L)$ and $\CC(X_\J,Y_\J)$ contain discontinuous functions.

\begin{proposition}[Function spaces for low functions]
\label{prop:function-space-low}
For instance for $X=Y=\IN^\IN$ the embeddings $\id:\CC(X,Y)^\L\to\CC(X^\L,Y^\L)$
and $\id:\CC(X,Y)_\J\to\CC(X_\J,Y_\J)$ are not computable.
The partial inverses of these maps are also not computable. 
\end{proposition}
\begin{proof}
The function $F$ from Proposition~\ref{prop:low-relative-low}
is not a computable point in $\CC(X^\L,Y^\L)$, but
by Corollary~\ref{cor:relative-low} it is a computable point in $\CC(X,Y)^\L$. 
This means that the embedding $\id:\CC(X,Y)^\L\to\CC(X^\L,Y^\L)$
is not computable. With Theorem~\ref{thm:jumps-products-exponentials} 
it follows that also the embedding  $\id:\CC(X,Y)_\J\to\CC(X_\J,Y_\J)$
is not computable since $Z^\L=(Z_\J)'$. 
The function $f$ from Proposition~\ref{prop:low-relative-low} can easily
be converted into a function $f:\IN^\IN\to\IN^\IN$ with corresponding properties, 
and hence $\CC(X^\L,Y^\L)$ contains computable points which are not 
computable in $\CC(X,Y)^\L$. 
The space $\CC(X,Y)_\J$ does not even contain any computable points whatsoever. 
Altogether, this shows that the partial inverses of the given maps are not computable.
\end{proof}

The positive properties that apply to Turing jumps and products
are captured in the following result.

\begin{proposition}[Turing jumps and products]
\label{prop:integral-product}
Let $X$ and $Y$ be represented spaces. Then the following maps are computable:
\begin{enumerate}
\item $\id:(X\times Y)_\J\to X_\J\times Y_\J$,
\item $\id:(X^\IN)_\J\to (X_\J)^\IN$. 
\end{enumerate}
In general, these maps are not computable isomorphisms.
\end{proposition}
\begin{proof}
It is easy to see that $\langle\J\times\J\rangle$ is limit computable since $\J$ is limit computable.
Hence, by the jump normal form theorem (Theorem~\ref{thm:jump}) there exists a computable $F$ such that
$F\J\langle p,q\rangle=\langle\J(p),\J(q)\rangle$.
This function $F$ is a realizer of $\id:(X\times Y)_\J\to X_\J\times Y_\J$.
Likewise, limit computability of $\langle\widehat{\J}\rangle$ implies that 
$\id:(X^\IN)_\J\to (X_\J)^\IN$ is computable. 
Let us assume that the inverse of $\id:(X\times Y)_\J\to X_\J\times Y_\J$ is computable.
Then by Proposition~\ref{prop:jump-endofunctor} and Fact~\ref{fact:function-space} the evaluation $\ev:\CC(X,Y)_\J\times X_\J\to Y_\J$ would be computable,
and hence $\id:\CC(X,Y)_\J\to\CC(X_\J,Y_\J)$ would be computable, which
is not the case in general by Proposition~\ref{prop:function-space-low}, for instance for $X=Y=\IN^\IN$.
If the inverse of $\id:(X\times Y)_\J\to X_\J\times Y_\J$ for binary products of $X=Y=\IN^\IN$ is not computable, 
it follows that the inverse of $\id:(X^\IN)_\J\to (X_\J)^\IN$ for countable products of $X=\IN^\IN$ is not computable  either.
\end{proof}

We mention that the jump $X\mapsto X_\J$ commutes with coproducts while
the jump $X\mapsto X'$ does not. However, we do not needs these facts here.
Hence, we leave the details to the reader. Here we just formulate a corollary for $X\mapsto X^\L$
that we obtain from Proposition~\ref{prop:integral-product}
and Theorem~\ref{thm:jumps-products-exponentials}.

\begin{corollary}[Lowness and products]
\label{cor:low-product}
Let $X$ and $Y$ be represented spaces. Then the following maps are computable:
\begin{enumerate}
\item $\id:(X\times Y)^\L\to X^\L\times Y^\L$,
\item $\id:(X^\IN)^\L\to (X^\L)^\IN$. 
\end{enumerate}
In general, these maps are not computable isomorphisms.
\end{corollary}

That the maps are not computable isomorphisms follows again from Spector's jump inversion theorem.
As a consequence of Corollary~\ref{cor:low-product} and Fact~\ref{fact:function-space} 
we obtain at least one positive implication for sequences that are computable
relative to a low oracle. The inverse implication is not correct in general by Proposition~\ref{prop:low-relative-low}.

\begin{corollary}[Low sequences]
\label{cor:low-sequences}
The identity $\id:\CC(\IN,X)^\L\to\CC(\IN,X^\L)$ is computable.
In particular, every function $f:\IN\to X$ that is computable relative to a low oracle
is also low as a function.
\end{corollary}

In the remainder of this section we briefly want to discuss some interpretations of our results
in the lattice of representations. We recall that given two represented spaces $(X,\delta_X)$ and $(Y,\delta_Y)$
with $X\In Y$ we say that $\delta_X$ is \emph{reducible} to $\delta_Y$, in symbols $\delta_X\leq\delta_Y$, 
if $\id:(X,\delta_X)\to(Y,\delta_Y)$ is computable. The corresponding equivalence is denoted by $\equiv$.

As an immediate corollary of Proposition~\ref{prop:jump-endofunctor} we obtain the following.

\begin{corollary}[Monotonicity of jumps]
\label{cor:jump-monotonicity}
Let $(X,\delta_X)$ and $(Y,\delta_Y)$ be represented spaces and let $T:\In\IN^\IN\to\IN^\IN$ be transparent and surjective.
Then $\delta_X\leq\delta_Y$ implies $\delta_X^T\leq\delta_Y^T$.
\end{corollary}

We now consider the jump operations that we have introduced on represented spaces.
Firstly, they are all ordered in the following way.

\begin{corollary}[Order of jumps]
\label{prop:order-operations}
For every representation $\delta$ each of the following reductions holds:
$\delta_\J\leq\delta_\HP\leq\delta\leq\delta^\Delta\leq\delta^\L\leq\delta'$.
\end{corollary}

Here the first reduction holds since $\lim$ has a computable right inverse $F:\IN^\IN\to\IN^\IN$, and hence
$F:(\IN^\IN)_\J\to(\IN^\IN)_\J$ is computable by Proposition~\ref{prop:jump-endofunctor}; the second reduction holds since $\HP^{-1}$ is computable;
the third reduction holds since $\lim_\Delta$ has a computable right inverse; the fourth reduction holds in $\lim_\Delta$ is low~by Corollary~\ref{cor:finite-low} and Theorem~\ref{thm:discrete-limit},
and the last reduction holds since $\L$ is limit computable by Proposition~\ref{prop:composition-low}.
In fact, some of the discussed operations on represented spaces have further properties.
In particular the pair $(\delta\mapsto\delta_\J,\delta\mapsto\delta')$ forms a \emph{Galois connection} in the lattice of representations
of a fixed set. This is a consequence of Theorem~\ref{thm:limit-computability} applied to the injection $\id:X\to Y$.

\begin{corollary}[Galois connection between Turing jumps and limits]
\label{cor:Galois}
Let $(X,\delta_X)$ and $(Y,\delta_Y)$ be represented spaces with $X\In Y$.
Then $\delta_{X_\J}\leq\delta_Y\iff\delta_X\leq\delta_{Y'}$.
\end{corollary}

The operation $\delta\mapsto\delta^\L=(\delta_\J)'$ is the \emph{monad} of this Galois connection, which 
implies that it is a closure operator. Likewise, $\delta\mapsto\delta_\HP=(\delta')_\J$ is an interior operator~\cite[Proposition~3(4)]{EKMS93}.

\begin{proposition}[Closure and interior operators]
\label{prop:discrete-low-closure}
$\delta\mapsto\delta^\L$ and $\delta\mapsto\delta^\Delta$ are closure operators
and $\delta\mapsto\delta_\HP$ is an interior operator
on the lattice of representations of a fixed set $X$.
\end{proposition}

The fact that $\delta\mapsto\delta^\Delta$ is also a closure operator can easily be proved directly.
Another consequence of Corollary~\ref{cor:Galois} which can also easily be proved directly,
is that $\delta\mapsto\delta_\J$ preserves suprema and $\delta\mapsto\delta'$ preserves infima~\cite[Proposition~3(8)]{EKMS93}.
We recall for represented spaces $(X,\delta_X)$ and $(Y,\delta_Y)$ the \emph{meet} or \emph{infimum} $\delta_X\wedge\delta_Y$, 
which is a representation of $X\cap Y$, can be defined by $(\delta_X\wedge\delta_Y)\langle p,q\rangle=z:\iff\delta_X(p)=\delta_Y(q)=z$.
If $X=Y$, then $\delta_X\wedge\delta_Y$ is actually known to the the infimum in the lattice of representations of $X$~\cite[Lemma~3.3.8]{Wei00}.
Hence we obtain the following.

\begin{proposition}[Infimum and jumps]
\label{prop:infima}
Let $(X,\delta_X)$ and $(Y,\delta_Y)$ be represented spaces. Then $(\delta_X\wedge\delta_Y)'\equiv\delta_X'\wedge\delta_Y'$.
\end{proposition}

Analogously, $\delta\mapsto\delta'$ commutes with products (see Theorem~\ref{thm:jumps-products-exponentials}),
which also implies Proposition~\ref{prop:infima} since $\delta_{X}\wedge\delta_{Y}=\Delta_{X\cap Y}^{-1}\circ\delta_{X\times Y}$,
where $\Delta_{X\cap Y}$ denotes the \emph{diagonal} $\Delta_{X\cap Y}:X\cap Y\to(X\cap Y)\times(X\cap Y),x\mapsto(x,x)$.
Likewise, $\delta\mapsto\delta_\J$ commutes with coproducts and suprema, but we do not formulate these results here.

\section{Limit computability on metric spaces}
\label{sec:metric}

In this section we want to transfer some of our results to computable metric spaces.
We recall that $(X,d,\alpha)$ is called a \emph{computable metric space}, if 
$(X,d)$ is a metric space with metric $d:X\times X\to\IR$ and $\alpha:\IN\to X$ is a sequence
that is dense in $X$ such that $d\circ(\alpha\times\alpha):\IN^2\to\IR$ is a computable sequence of real numbers.
Each computable metric space is equipped with its \emph{Cauchy representation} $\delta_X:\In\IN^\IN\to X$ 
that is defined by $\delta_X(p):=\lim_{n\to\infty}\alpha p(n)$ with $\dom(\delta_X):=\{p\in\IN^\IN:(\forall k)(\forall n\geq k)\;d(\alpha p(n),\alpha p(k))<2^{-k}$
and $(\alpha p(n))_{n\in\IN}$ converges$\}$. The real numbers $\IR$ are also equipped with a Cauchy representation
that is induced by a standard numbering $\alpha$ of all rational numbers. 
Besides the usual Cauchy representation there is also the so-called \emph{naive Cauchy representation} $\delta_X^\mathrm{n}$
that is defined exactly as $\delta_X$ but with $\dom(\delta_X^\mathrm{n}):=\{p\in\IN^\IN:(\alpha p(n))_{n\in\IN}$ converges$\}$.

Given a computable metric space $(X,d,\alpha)$ we define the \emph{open ball} $B_{\langle c,r\rangle}:=B(\alpha(c),\overline{r}):=\{x\in X:d(x,\alpha(c))<\overline{r}\}$
for every $c,r\in\IN$,
where $\overline{\langle i,j,k\rangle}:=\frac{i-j}{k+1}$ denotes the rational number encoded by $i,j,k\in\IN$.
By $W_i:=\{n\in\IN:$ Turing machine $i$ halts on input $n\}$ we denote the usual numbering of c.e.\ subsets of $\IN$,
and by $U_i:=\bigcup_{n\in W_i}B_n$ we denote a numbering of all \emph{c.e.\ open} subsets of $X$.
We can consider the set $\OO(X)$ of open subsets as a represented space with representation
$\delta_{\OO(X)}(p)=\bigcup_{i\in\IN}B_{p(i)}$. The c.e.\ open subsets are exactly the computable points in $\OO(X)$
with this representation.

We can now generalize the limit map and the Turing jump as defined in Example~\ref{example:limit-computable}
to arbitrary computable metric spaces.

\begin{definition}[Limit and Turing jump]
\label{def:limit-jump}
Let $X$ be a computable metric space. We define:
\begin{enumerate}
\item The \emph{limit map} of $X$ by
$\lim\nolimits_X:\In X^\IN\to X,(x_n)_{n\in\IN}\mapsto\lim_{n\to\infty}x_n$.
\item The \emph{Turing jump} $\J_X:X\to\IN^\IN$ of $X$ by
\[\J_X(x)(i):=\left\{\begin{array}{ll}
1 & \mbox{if $x\in U_i$}\\
0 & \mbox{otherwise}
\end{array}\right.
\]
for all $x\in X$ and $i\in\IN$.
\end{enumerate}
\end{definition}

It is easy to see that both functions defined here are always limit computable.

\begin{proposition}
\label{prop:metric-lim-jump}
$\lim_X$ and $\J_X$ are limit computable for every computable metric space $X$.
\end{proposition}
\begin{proof}
Let $\delta_X$ be the Cauchy representation of the computable metric space $(X,d,\alpha)$.
Given a $p=\langle p_0,p_1,p_2,...\rangle\in\IN^\IN$ with a sequence $(p_n)_{n\in\IN}$ of names of points $x_n:=\delta_X(p_n)$ such that
$x:=\lim_{n\to\infty}x_n$ exists, we need to devise a limit computation that yields a $q\in\dom(\delta_X)$
with $\delta_X(q)=x$. For every $i,k\in\IN$ the property
\[(\exists n>i)\;d(\delta_X(p_n),\delta_X(p_i))>2^{-k-2}\]
is c.e.\ in $p$ and $k,i$. Hence, with the help of $\J(p)$ we can decide this property,
and hence for each $k\in\IN$ we can systematically search for a $i=i_k\in\IN$ such that it fails. Such 
an $i_k$ must exists since $x_n$ converges to $x$. We can assume that the sequence $(i_k)_{k\in\IN}$ is strictly
monotone increasing.
For each $k\in\IN$ we let $q(k):=p_{i_k}(k+2)$. Then for $n>k$
\begin{eqnarray*}
d(\alpha q(n),\alpha q(k))&\leq& d(\alpha p_{i_n}(n+2),\delta_X(p_{i_n}))+d(\delta_X(p_{i_n}),\delta_X(p_{i_k}))+d(\delta_X(p_{i_k}),\alpha p_{i_k}(k+2))\\
&\leq& 2^{-n-2}+2^{-k-2}+2^{-k-2}<2^{-k}
\end{eqnarray*}
and $\delta_X(q)=\lim_{k\to\infty}\alpha p_{i_k}(k+2)=\lim_{k\to\infty}x_k=x$.
By Theorem~\ref{thm:jump} it follows that $\lim_X$ is limit computable.

$\J_X$ can easily be computed on a limit Turing machine. Given a name of an input $x\in X$ one produces for each
output position $\J_X(x)(i)$ the default value $0$ and simultaneously one tries to verify $x\in U_i$. As soon 
as $x\in U_i$ can be confirmed, the $i$--th output value has to be changed to $1$, which happens at most once
for each position $i$.
\end{proof}

This allows us to transfer the limit normal form theorem (Theorem~\ref{thm:limit}) to computable metric spaces.
For $f,s:\In X\mto Y$ we write $s\prefix f$ if $s(x)\In f(x)$ for all $x\in\dom(f)$ and $\dom(f)\In\dom(s)$.
In this situation we call $s$ a \emph{solution} or \emph{selector} of $f$.

\begin{theorem}[Limit normal form]
\label{thm:metric-limit-normal-form}
Let $X$ be a represented space, $Y$ a computable metric space and
$f:\In X\mto Y$ a problem. Then the following are equivalent:
\begin{enumerate}
\item $f:\In X\mto Y$ is limit computable,
\item $\lim_Y\circ g\prefix f$ for some computable $g:\In X\mto Y^\IN$.
\end{enumerate}
\end{theorem}
\begin{proof}
That (2) implies (1) follows since $\lim_Y$ is limit computable by Proposition~\ref{prop:metric-lim-jump}.
We need to prove that (1) implies (2). Let $f:\In X\mto Y$ be limit computable.
We use the Cauchy representation $\delta_Y$ of $Y$.
Then $f$ has a limit computable realizer $F:\In\IN^\IN\to\IN^\IN$.
By the limit normal form theorem (Theorem~\ref{thm:limit}) there exists a computable $G:\In\IN^\IN\to\IN^\IN$
such that $F=\lim\circ G$.  Now we define a function $H:\In\IN^\IN\to\IN^\IN$
by $H(p)=\langle r_0,r_1,r_2,...\rangle$, provided that $G(p)=\langle q_0,q_1,q_2,...\rangle$, where each $r_i$
is determined as follows. We choose $r_1:=r_0:=(q_0(0),q_0(0),...)$ and if $i>1$ then
for each $k\in\{2,...,i\}$ we check, which of the following two conditions is recognized first:
\begin{enumerate}
\item $(\forall j\in\{1,...,k-1\})\;d(\alpha q_i(k),\alpha q_i(j))<2^{-j+1}$,
\item $(\exists j\in\{1,...,k-1\})\;d(\alpha q_i(k),\alpha q_i(j))>2^{-j}$.
\end{enumerate}
Depending on which condition is recognized first, we choose 
\begin{enumerate}
\item $r_i:=(q_i(1),q_i(2),...,q_i(i),q_i(i),....)$, if for all $k\leq i$ the first condition is recognized first,
\item $r_i:=(q_i(1),q_i(2),...,q_i(k-1),q_i(k-1),...)$, if $k\leq i$ is minimal such that the second condition is recognized first.
\end{enumerate}
We note that $r_i\in\dom(\delta_Y)$ for all $i$ and if $q_i$ has a prefix of length $k\leq i$ that is a valid prefix
of a name in $q\in\dom(\delta_Y)$, then $d(\delta_Y(r_i),\delta_Y(q))\leq 2^{-k+1}$.
Altogether, $H$ is computable, and the definition ensures that  $\range(H)\In\dom(\delta_Y^\IN)$
and $\lim_Y\circ\delta_Y^\IN\circ H=\delta_Y\circ\lim\circ G=\delta_Y\circ F$.
This proves that $H$ is a realizer of a computable multi-valued  function $g:\In X\mto Y^\IN$ with $\lim_Y\circ g\prefix f$.
\end{proof}

For completeness we formulate the result also for single-valued $f$.
We note that the problem $g$ still needs to be multi-valued in general.

\begin{corollary}[Limit normal form]
\label{cor:metric-limit-normal-form}
Let $X$ be a represented space, $Y$ a computable metric space and
$f:\In X\to Y$ a map. Then the following are equivalent:
\begin{enumerate}
\item $f:\In X\to Y$ is limit computable,
\item $f=\lim_Y\circ g$ for some computable $g:\In X\mto Y^\IN$.
\end{enumerate}
\end{corollary}

As a corollary we obtain the following conclusion on pointwise limits of computable sequences of functions.

\begin{corollary}[Pointwise limit]
\label{cor:pointwise-metric}
Let $X$ be a represented space and $Y$ a computable metric space.
Let $(g_n)_{n\in\IN}$ be a computable sequence in $\CC(X,Y)$, i.e., a computable
sequence of computable functions $g_n:X\to Y$. Then $f:X\to Y$, defined by 
$f(x):=\lim_{n\to\infty}g_n(x)$ for all $x\in X$, is limit computable.
\end{corollary}

We note that in general one cannot expect that all limit computable $f$ can be obtained as the pointwise
limit of a computable sequence of single-valued $g$ (as in the case of Baire space, see Corollary~\ref{cor:pointwise}).
We obtain, however, a characterization of limit computable points in computable metric spaces.

\begin{corollary}[Limit computable points]
\label{cor:metric-limit-computable-points}
Let $X$ be a computable metric space. Then $x\in X$ is limit computable (in the sense that it has a 
limit computable name) if and only if there is a computable sequence $(x_n)_{n\in\IN}$ such that 
$x=\lim_{n\to\infty}x_n$.
\end{corollary}

We recall that a computable metric space $(X,d)$ is called \emph{computably compact} if $\{X\}$ is c.e.\ open in $\OO(X)$.
We note that for computably compact metric spaces $(X,d)$ the function space $\CC(X)=\CC(X,\IR)$ with the metric induced by the \emph{uniform norm} $||f||:=\sup_{x\in X}|f(x)|$ 
is a computable metric space again. As a dense subset one can use, for instance, the set of rational polynomials
in the \emph{distance functions} $d_x:X\to\IR,y\mapsto d(x,y)$.
This space $\CC(X)$ is then computably isomorphic to the space obtained by using the function space representation $\delta_{\CC(X,\IR)}$,  see \cite{Bra98b}.
If we apply Corollaries~\ref{cor:limit-function-points} and \ref{cor:metric-limit-computable-points} to the space $\CC(X)$, then we obtain the following result.

\begin{corollary}[Uniform limit]
\label{cor:uniform-metric}
Let $X$ be a computably compact computable metric space.
Then $f:X\to\IR$ is computable relative to the halting problem if and only if there is a computable sequence $(f_n)_{n\in\IN}$
of computable functions $f_n:X\to\IR$ that converges uniformly to $f$, i.e., such that $\lim_{n\to\infty}||f-f_n||=0$.
\end{corollary}

A characterization similar to Corollary~\ref{cor:metric-limit-normal-form} was obtained
for effectively $\SO{2}$--measurable functions in \cite[Theorem~9.5]{Bra05} (we refer the reader
to \cite{Bra05} for the definition of measurability, as we are not going to use these notions here any further). 
This yields the following conclusion.

\begin{corollary}[Borel measurability]
\label{cor:Borel}
Let $X,Y$ be computable metric spaces. Then 
a function $f:X\to Y$ is limit computable if and only if it is effectively $\SO{2}$--measurable.
\end{corollary}

Corollary~\ref{cor:metric-limit-normal-form} can also be applied to the naive Cauchy representation
since $\delta_X^\mathrm{n}=\lim_X\circ\widehat{\alpha}$, where $\widehat{\alpha}:\IN^\IN\to X^\IN,p\mapsto (\alpha p(0),\alpha p(1),\alpha p(2),...)$
is the parallelization of $\alpha$ that is computable with respect to the Cauchy representation $\delta_X$.
We obtain the following result that generalizes corresponding results for the spaces $X=\IR$
and $X=\CC([0,1]^n)$ by Ziegler~\cite[Propositions~2.5 and 2.7]{Zie07}.

\begin{corollary}[Naive Cauchy representation as jump]
\label{cor:naive}
$\delta_X^\mathrm{n}\equiv\lim_X\circ\delta_X^\IN\equiv\delta_X'$ for every computable metric space $X$.
\end{corollary}
\begin{proof}
With the help of Theorem~\ref{thm:limit-computability} and Corollary~\ref{cor:metric-limit-normal-form} 
one obtains $\delta_X^\mathrm{n}\leq\delta_X'$.
The reduction $\delta_X'\leq\lim_X\circ\delta_X^\IN$ follows from the same theorems applied to 
the identity $\id:X'\to X$.
The remaining reduction $\lim_X\circ\delta_X^\IN\leq\lim_X\circ\widehat{\alpha}$ is easy to prove directly,
by choosing a diagonal sequence.
\end{proof}

Together with Theorem~\ref{thm:computability-halting-problem} we obtain a proof 
of the following result from~\cite[Theorem~22]{BH02}.

\begin{corollary}[B.\ and Hertling 2002]
Let $X,Y$ be computable metric spaces.
The functions $f:\In X\to Y$ that are continuous with respect to the naive Cauchy representations
on $X$ and $Y$ are the usual continuous functions.
\end{corollary}

Now we study the jump $\J_X$. We first prove that $\J_X^{-1}$ is computable in general.

\begin{proposition}[Inverse jump]
\label{prop:inverse-jump-metric}
$\J_X^{-1}:\In\IN^\IN\to X$ is computable for every computable metric space $X$.
\end{proposition}
\begin{proof}
We recall that a standard representation of $X$ can be defined by 
$\delta(p)=x:\iff\range(p)=\{n\in\IN:x\in B_n\}$. It is easy to see that $\delta\leq\delta_X$,
where $\delta_X$ denotes the Cauchy representation of $X$~\cite[Theorem~8.1.4]{Wei00}.
Hence, it suffices to compute $\J_X^{-1}$ with respect to $\delta$.
Since $\J_X(x)$ is the characteristic function of $\{n\in\IN:x\in U_n\}$, it is straightforward
to generate a list of all $n\in\IN$ with $x\in B_n$, given $\J_X(x)$.
\end{proof}

Using this result we obtain the following form of the jump normal form theorem.
Surprisingly we need a variant of Schr\"oder's representation~\cite{Sch95} 
of computable metric spaces that has compact fibers, which can be computed 
with respect to positive information. 
Such a representation was studied in~\cite{Bra05}.

\begin{theorem}[Jump normal form]
\label{thm:jump-metric}
Let $X,Y$ be computable metric spaces and
$f:\In X\mto Y$ a problem. Then the following are equivalent:
\begin{enumerate}
\item $f:\In X\mto Y$ is limit computable,
\item $g\circ\J_X\prefix f$ for some computable $g:\In \IN^\IN\to Y$.
\end{enumerate}
\end{theorem}
\begin{proof}
By Proposition~\ref{prop:metric-lim-jump} it is clear that $g\circ\J_X$ is limit computable for every computable $g:\In\IN^\IN\to Y$.
Let now $f:\In X\mto Y$ be limit computable. By Theorem~\ref{thm:metric-limit-normal-form} there is a computable
$h:\In X\mto Y^\IN$ such that $\lim_Y\circ h\prefix f$. 
Let $\delta\equiv\delta_X$ be a representation of $X$ that is equivalent to the Cauchy representation $\delta_X$
and that makes $\kappa_X:X\to\KK_+(\IN^\IN),x\mapsto\delta^{-1}\{x\}$ computable, where
$\KK_+(\IN^\IN)$ denotes the space of compact subsets of $\IN^\IN$ endowed with the
representation $\delta_{\KK_+(\IN^\IN)}(p)=K:\iff \range(p)=\{n\in\IN:K\cap B_n\not=\emptyset\}$. 
Such a representation $\delta$ with compact fibers $\delta^{-1}\{x\}$ exists for every computable metric space~\cite[Lemma~6.3~(5)]{Bra05}. 
Let $F$ be a computable realizer of $h$, i.e., $\lim_Y\delta_{Y^\IN} F\prefix f\delta$. 
That $\kappa_X$ is computable implies that the set 
\[U:=\{(x,n,k)\in X\times\IN^2:(\exists p\in\delta^{-1}\{x\})(\exists i>n)\;d_Y(\delta_Y(F(p)(i)),\delta_Y(F(p)(n)))>2^{-k}\}\]
is c.e.\ open relative to $\dom(f)\times\IN^2$, and hence there is a computable function $r:\IN\to\IN$ such that
$U_{r\langle n,k\rangle}\cap\dom(f)=\{x\in X:(x,n,k)\in U\}\cap\dom(f)$.
Let $x\in\dom(f)$ and $k\in\IN$. Then for all $p\in\delta^{-1}\{x\}$ there is some $n\in\IN$ such that 
for all $i>n$ we have $d_Y(\delta_Y(F(p)(i)),\delta_Y(F(p)(n)))\leq2^{-k}$, because $(\delta_YF(p)(n))_{n\in\IN}$ is convergent.
Since $\delta^{-1}\{x\}$ is compact, the number $n\in\IN$ even exists uniformly for all $p\in\delta^{-1}\{x\}$.
That is, for every $x\in\dom(f)$ and $k\in\IN$ there is some $n\in\IN$ with $x\not\in U_{r\langle n,k\rangle}$.
With the help of $\J_X(x)(r\langle n,k\rangle)$ we can actually check this condition, and hence we can compute an $s:\IN\to\IN$ such that
for each $k\in\IN$ we obtain $x\not\in U_{r\langle s(k),k\rangle}$.
By Proposition~\ref{prop:inverse-jump-metric} we can also obtain a $p\in\IN^\IN$ with $\delta(p)=x$
from $\J_X(x)$. Hence there is a computable $g:\In\IN^\IN\to Y$ with $g\circ\J_X(x)=\lim_{k\to\infty}\delta_Y(F(p)(s(k)))=\lim_Y\delta_{Y^\IN}F(p)\prefix f(x)$ 
for all $x\in\dom(f)$, where the first limit is computable since it is effective (with speed of convergence $2^{-k}$, see the proof of Proposition~\ref{prop:metric-lim-jump}
for a detailed calculation).
\end{proof}

Since $\J_X$ is injective, we can assume that $\dom(g)$ is such that $\dom(g\circ\J_X)=\dom(f)$.
We formulate the characterization for single-valued functions as a corollary.

\begin{corollary}[Jump normal form]
\label{cor:jump-metric}
Let $X,Y$ be computable metric spaces and
$f:\In X\to Y$ a function. Then the following are equivalent:
\begin{enumerate}
\item $f:\In X\to Y$ is limit computable,
\item $f=g\circ\J_X$ for some computable $g:\In \IN^\IN\to Y$.
\end{enumerate}
\end{corollary}

Another consequence of Theorem~\ref{thm:jump-metric} is that every multi-valued
limit computable problem has a single-valued limit computable selector.

\begin{corollary}[Selection]
\label{cor:selection}
Let $X,Y$ be computable metric spaces. 
For every limit computable $f:\In X\mto Y$ there is a limit computable single-valued 
$g:\In X\to Y$ with $g\prefix f$.
\end{corollary}

Now we define $1$--genericity for computable metric spaces.

\begin{definition}[$1$--genericity]
Let $X$ be a computable metric space. Then we call $x\in X$ \emph{$1$--generic} if it is a point
of continuity of $\J_X$.
\end{definition}

It is easy to see that $1$--generic points can be characterized as those points that avoid
the boundaries of all c.e.\ open sets. By $\partial U$ we denote the \emph{boundary} of $U$.

\begin{lemma}[$1$--genericity]
\label{lem:genericity}
Let $X$ be a computable metric space. Then $x\in X$ is $1$--generic if and only if 
$x\in X\setminus\bigcup_{i=0}^\infty \partial U_i$.
\end{lemma}

This characterization is equivalent to the definition of $1$--genericity as it is used in computability theory
and as it has been used in computable metric spaces before (see \cite[Corollary~2.3]{KT14}, \cite[Lemma~9.2]{BHK16} or \cite[Definition~2.1]{Hoy17}). 
By the Baire category theorem every complete metric space is comeager and since the set $\bigcup_{i=0}^\infty\partial U_i$ is meager,
it follows that the $1$--generics form a comeager $G_\delta$--set in any complete computable metric space.
In particular, in a complete computable metric space, there are many $1$--generic points.
We obtain the following characterization of $1$--genericity.

\begin{theorem}[$1$--genericity]
\label{thm:characterization-genericity-metric}
Let $X$ be a computable metric space. 
Then $x\in X$ is $1$--generic if and only if for every computable metric space $Y$, every limit computable $f:\In X\to Y$  with $x\in\dom(f)$ is continuous at $x$.
\end{theorem}
\begin{proof}
For the ``if'' direction it is sufficient to consider $Y=\IN^\IN$ and $f=\J_X:X\to\IN^\IN$. For the ``only if'' direction
we assume that $f:\In X\to Y$ is limit computable and $x\in\dom(f)$. Then by Corollary~\ref{cor:jump-metric}
there is a computable $g:\In\IN^\IN\to X$ such that $f=g\circ\J_X$. Hence, it follows that $f$ is continuous at $x$.
\end{proof}

The ``if'' direction of this theorem can be seen as a computability theoretic version of a well-known theorem of Baire that
states that every Baire class $1$ function has a comeager $G_\delta$--set of continuity points~\cite[Theorem~24.14]{Kec95}.
We note that the ``only if'' direction of Theorem~\ref{thm:characterization-genericity-metric} 
for the special case of $X=[0,1]$ was also proved by Kuyper and Terwijn~\cite[Theorem~4.2]{KT14}.

Since $\lim_X$ is not continuous at any point for computable metric spaces $X$ with at least two points, 
it follows that no $1$--generic sequence in such a space is convergent.

\begin{corollary}[$1$--generic sequences]
Let $X$ be a computable metric space with at least two points. Then 
every sequence $(x_n)_{n\in\IN}$ that is a $1$--generic point in $X^\IN$ 
is not convergent.
\end{corollary}

The following result generalizes Proposition~\ref{prop:jump-1-generic} to computable metric spaces.

\begin{proposition}[Jump on $1$--generics]
\label{prop:jump-1-generic-metric}
The Turing jump operator $\J_X:X\to\IN^\IN$ restricted to the set  
of $1$--generic points $x\in X$ is computable relative to the halting problem for all computable metric spaces $X$.
\end{proposition}
\begin{proof}
We consider the metric space $(X,d)$ and
we define the relations of \emph{formal inclusion} and \emph{formal disjointness} by
\begin{enumerate}
\item $B_{\langle c_1,r_1\rangle}\lhd B_{\langle c_2,r_2\rangle}:\iff d(\alpha(c_1),\alpha(c_2))+\overline{r_1}<\overline{r_2}$,
\item $B_{\langle c_1,r_1\rangle}\bowtie B_{\langle c_2,r_2\rangle}:\iff d(\alpha(c_1),\alpha(c_2))>\overline{r_1}+\overline{r_2}$. 
\end{enumerate}
These are rather relations between the numbers $c_1,r_1\in\IN$ and $c_2,r_2\in\IN$ than between the corresponding balls,
but they are denoted for convenience in the given way. It is clear that both relations are c.e.\ and that
$B_j\lhd B_i$ implies $B_j\In B_i$ and $B_j\bowtie B_i$ implies $B_j\cap B_i=\emptyset$.
The sets
\begin{enumerate}
\item $A:=\{(n,j)\in\IN^2:(\exists i\in W_n)\;B_j\lhd B_i\}$,
\item $B:=\{(n,j)\in\IN^2:(\forall i\in W_n)\;B_j\bowtie B_i\}$
\end{enumerate}
are both computable relative to $0'$. A $1$--generic $x\in X$ satisfies $x\in X\setminus\bigcup_{n=0}^\infty\partial U_n$ and hence
\begin{enumerate}
\item $x\in U_n\iff (\exists j\in\IN)(x\in B_j$ and $(n,j)\in A)$,
\item $x\not\in U_n\iff (\exists j\in\IN)(x\in B_j$ and $(n,j)\in B)$.
\end{enumerate}
Thus, for each $1$--generic $x$ we can systematically search for some $j\in\IN$ such that $x\in B_j$ and either
$(n,j)\in A$ or $(n,j)\in B$ and the condition that holds tells us how to compute $\J_X(x)(n)$.
Hence, $\J_X$ restricted to $1$--generics is computable relative to $0'$.
\end{proof}

Another consequence of the jump normal form (Theorem~\ref{thm:jump-metric}) together with Proposition~\ref{prop:jump-1-generic-metric}
is the following result that generalizes Corollary~\ref{cor:jump-1-generic}.

\begin{corollary}[Limit computability on $1$--generics]
\label{cor:jump-1-generic-metric}
Restricted to $1$--generics every limit computable $F:\In X\to Y$ on computable metric spaces $X,Y$ 
is computable relative to the halting problem.
\end{corollary}

Now the question appears how the property of being $1$--generic relates to the property of having
a $1$--generic name. What we can prove in general is that atypical points have atypical names.
As a preparation we need to show that admissible representations are hereditarily quotient maps.
We recall that a representation $\delta$ of a topological space $X$ is called \emph{admissible} (with respect to this space) 
if and only if it is maximal with respect to the topological version $\leqt$ of the reducibility $\leq$ among all continuous representations~\cite{Sch02}.
The Cauchy representation $\delta_X$ is an example of an admissible representation~\cite[Theorem~8.1.4]{Wei00}, and hence any representation $\delta\equiv\delta_X$
is also admissible. It follows from the second half of the proof of \cite[Theorem~4]{Sch02} that admissible representations
are hereditary quotient maps in the following sense.

\begin{lemma}[Hereditary quotients]
\label{lem:admissibility}
Let $\delta$ be an admissible representation of a topological space $X$, and let $Y$ be another topological space.
Let $f:\In X\to Y$ be a function and $x\in\dom(f)$. If $f\delta$ is continuous at all $p\in\delta^{-1}\{x\}$, then
$f$ is sequentially continuous at $x$. 
\end{lemma}

Since every metric space $X$ is sequential, we can replace sequential continuity by ordinary continuity in this case.
Then the property expresses what is usually called a {\em hereditary quotient map} or a {\em pseudo-open map}~\cite{AP90}.

\begin{corollary}[Hereditary quotients]
Every admissible representation of a sequential topological space is a 
hereditary quotient map (or equivalently, is pseudo-open).
\end{corollary}

Using Lemma~\ref{lem:admissibility} we can now prove the following result, which shows that atypical
points have atypical names with respect to every representation in the equivalence class of the
Cauchy representation.

\begin{proposition}[Non-generic names]
\label{prop:non-generic-names}
Let $X$ be a computable metric space with Cauchy representation $\delta_X$ and $x\in X$. 
Let $\delta\equiv\delta_X$ be another representation of $X$. 
If $x$ is not $1$--generic, then there is a $p\in\IN^\IN$ with $\delta(p)=x$ that is not $1$--generic.
\end{proposition}
\begin{proof}
By Propositions~\ref{prop:metric-lim-jump} and \ref{thm:jump} there is some computable $F$ such that $\J_X\delta=F\J$. 
Let $x$ be such that all names $p\in\delta^{-1}\{x\}$ are $1$--generic. 
Then $F\J$ is continuous at all these $p$ and hence so is $\J_X\delta$.
By Lemma~\ref{lem:admissibility} this implies that $\J_X$ is continuous at $x$, 
hence $x$ is $1$--generic.
\end{proof}

This result captures as much as can be said about the relation of arbitrary points 
and their names with respect to genericity in arbitrary computable metric spaces.
It is easy to see that in the equivalence class of the Cauchy representation $\delta_X$ there
are always representations without $1$--generic names: one can for instance define $\delta\equiv\delta_X$
such that names are always zero in every second component and hence not $1$--generic.
But even if we restrict ourselves to total $\delta\equiv\delta_X$ for complete metric spaces
$X$, then there are spaces such as $X=\IN$ in which all points are $1$--generic.
Hence having a non $1$--generic name does not imply being not $1$--generic in general.
If we consider the case of the reals $X=\IR$ and reasonable total variants of the Cauchy representation
(such as the one from \cite[Lemma~6.1]{BHK16}),
then there are syntactic properties of names such as: every component of the name encodes
a rational number with even denominator. This property constitutes a co-c.e.\ closed set $A\In\IN^\IN$
that is nowhere dense, i.e., $\partial A=A$, which means that all points in $A$ are not $1$--generic.
On the other hand, every real $x$, in particular every $1$--generic real,
has a name $p\in A$. Thus having a non $1$--generic name does not imply being not $1$--generic for 
such a representation.
For special types of spaces and special representations one
can provide further results along these lines. 
For instance Kuyper and Terwijn proved that an irrational number $x\in[0,1]$ is $1$--generic 
if and only if its unique binary expansion is $1$--generic~\cite[Proposition~2.5]{KT14}.

\section{Limit computability and computability relative to the halting problem}
\label{sec:limit-halting}

It is clear that every function that is computable relative to the halting problem is also limit computable.
This observation even holds uniformly in the sense specified in the following corollary.
The corollary is a consequence of Theorem~\ref{thm:jumps-products-exponentials}~(3) since $\id:X\to X'$ is computable.

\begin{corollary}[Limit computability of functions computable relative to the halting problem]
\label{cor:limit-halting}
$\id:\CC(X,Y)'\to\CC(X,Y')$ is computable for all represented spaces $X,Y$.
\end{corollary}

For some spaces, such as $X=\IN$, this identity is even a computable isomorphism (see Corollary~\ref{cor:sequences}).
However, it is clear that this does not happen for many other spaces $X,Y$ since often $\CC(X,Y)'\subsetneqq\CC(X,Y')$. 
The former set contains only continuous functions, the latter can contain discontinuous ones.
Even restricted to continuous functions the identity in Corollary~\ref{cor:limit-halting} is not a computable isomorphism in general.

We can, however, characterize the functions that are computable relative to the halting problems in terms
of limit computable functions. For this purpose we need the notion of a modulus of continuity.
Let $(X,d_X)$ and $(Y,d_Y)$ be metric spaces and $f:X\to Y$ a function.
Then $m:\IN\to\IN$ is called a \emph{modulus of continuity} of $f$ at $x\in X$ if
\[d_X(x,y)<2^{-m(n)}\TO d_Y(f(x),f(y))<2^{-n}\]
holds for all $y\in X$ and $n\in\IN$. 
The following result is well-known (see, for instance, \cite[Lemma~4.4.25]{Bra98b}).

\begin{proposition}[Modulus of continuity]
\label{prop:modulus}
Let $X,Y$ be computable metric spaces. Then
\[\Mod:\CC(X,Y)\times X\mto \IN^\IN,(f,x)\mapsto\{m\in\IN^\IN:m\mbox{ is a modulus of continuity of $f$ at $x$}\}\] 
is computable.
\end{proposition}

We call $M:X\mto\IN^\IN,x\mapsto\{m\in\IN^\IN:m$ is a modulus of continuity of $f$ at $x\}$ the 
\emph{global modulus of continuity} of $f$.
Using the previous proposition we can prove the following characterization.

\begin{theorem}[Computability relative to the halting problem]
\label{thm:limit-halting}
Let $X,Y$ be computable metric spaces and let $f:X\to Y$ be a function. Then the following are equivalent:
\begin{enumerate}
\item $f$ is computable relative to the halting problem,
\item $f$ is limit computable and continuous and its global modulus of continuity is computable relative to the halting problem.
\end{enumerate}
\end{theorem}
\begin{proof}
Let $f:X\to Y$ be computable relative to the halting problem. Then $f:X'\to Y'$ is computable by Theorem~\ref{thm:computability-halting-problem},
and hence $f$ is computable as a point in $\CC(X,Y)'$ by Theorem~\ref{thm:jumps-products-exponentials}. In particular, $f$ is also limit computable (see Corollary~\ref{cor:limit-halting}).
By Theorem~\ref{thm:jumps-products-exponentials} and Proposition~\ref{prop:jump-endofunctor} the problem $\Mod$ from Proposition~\ref{prop:modulus} is computable with type $\Mod:\CC(X,Y)'\times X'\mto(\IN^\IN)'$,
and hence the global modulus of $f$ is computable with type $M:X'\mto(\IN^\IN)',x\mapsto \Mod(f,x)$.  
By Theorem~\ref{thm:computability-halting-problem} this means that $M$ is computable relative to the halting problem.

Let now $f$ be limit computable and continuous with a global modulus of continuity $M$ that is computable relative to the halting problem.
We consider the computable metric spaces $(X,d_X,\alpha)$ and $(Y,d_Y)$.
If $f$ is limit computable, then also $f\circ\alpha:\IN\to Y$ is limit computable and hence computable relative to the halting problem by Corollary~\ref{cor:sequences}.
Given $x\in X$ and $k\in\IN$, we can compute with the help of the halting problem a modulus of continuity $m:\IN\to\IN$ of $f$ at $x$,  and we can compute an $n\in\IN$ with $d_X(\alpha(n),x)<2^{-m(k)}$, which implies $d_Y(f\alpha(n),f(x))<2^{-k}$. 
Since $f\alpha(n)$ can be computed with the help of the halting problem, we have found an approximation of $f(x)$ with precision $2^{-k}$ with the help of the halting problem.
This means that $f$ is computable relative to the halting problem. 
\end{proof}

This result tells us exactly which additional condition beyond continuity a limit computable function $f:X\to Y$ must satisfy in order to be computable relative to the halting problem. 
If a function is limit computable and just continuous on metric spaces, then 
a hyperarithmetical oracle is sufficient to compute it. 
In fact, we can prove an even stronger result.\footnote{This stronger result and its proof have been kindly
suggested by an anonymous referee.}

\begin{theorem}[Effective Borel measurability and continuity]
\label{thm:Borel-continuity}
Let $X,Y$ be computable metric spaces and let $X$ be complete. 
Let $A\In X$ be an effective $\Sigma^1_1$--subset and let $f:A\to Y$ be a function.
If $f$ is effectively Borel measurable and continuous, then $f$ is computable relative to
some hyperarithmetical oracle (i.e., an oracle in $\Delta^1_1$).
\end{theorem}
\begin{proof}
We use a standard numbering $(U_n^p)_{n\in\IN}$ of the sets $U_n^p\In Y$ that are c.e.\ relative to $p\in\IN^\IN$.
We omit the upper index in order to denote the unrelativized c.e.\ open subsets.
For any c.e.\ open set $U_n\In Y$ we have that $f^{-1}(U_n)$ is effectively $\Delta^1_1$ and open in $A$,
i.e., we can effectively find some $\Delta_1^1$--set $D_n$ in $X$ and there is an open set $V\In X$ such that $f^{-1}(U_n)=V\cap A=D_n\cap A$.
In particular, $f^{-1}(U_n)$ and $A\setminus f^{-1}(U_n)$ are effectively $\Sigma^1_1$ in $X$ and
separated by the open set $V$. Then by the Louveau separation theorem \cite[Theorem~B, page~369]{Lou80} the
sets $f^{-1}(U_n)$ and $A\setminus f^{-1}(U_n)$ are also separated by a $\Sigma^{0,p}_1$--set
for some hyperarithmetical $p\in\IN^\IN$ (i.e., by a c.e.\ open set relative to $p$).
Hence $f^{-1}(U_n)=U_k^p\cap A$ for some $k\in\IN$ and some hyperarithmetical $p\in\IN^\IN$.
Now we consider the function $g:\IN\to\IN^\IN,n\mapsto\langle k,p\rangle$. The property $B:=\{\langle n,k,p\rangle\in\IN^\IN:f^{-1}(U_n)=U_k^p\cap A\}$
is an effective $\Pi^1_1$--property, since
\[
\langle n,k,p\rangle\in B
\iff (\forall x\in X)((x\not\in U_k^p\cap A\vee x\in D_n)\wedge(x\not\in f^{-1}(U_n)\vee x\in U_k^p)).
\]
Hence, by the effective $\Pi^1_1$--uniformization theorem~\cite[Theorem~4E.4]{Mos09}
the function $g$ is a total $\Pi^1_1$--function, hence in $\Delta^1_1$. Clearly, $f$ is computable relative to $g$.
\end{proof}

As an immediate corollary we obtain the following result.

\begin{corollary}[Limit computability and continuity]
\label{cor:limit-continuity}
Let $X,Y$ be computable metric spaces and let $X$ be complete. 
If $f:X\to Y$ is limit computable and continuous,
then $f$ is computable relative to some hyperarithmetical oracle (i.e., an oracle in $\dI{1}$).
\end{corollary}

For uniformly continuous functions we can improve this result.
In this case even the jump of the halting problem suffices to compute the function.

\begin{theorem}[Limit computability and uniform continuity]
\label{thm:limit-uniform-continuity}
Let $X,Y$ be computable metric spaces. If $f:X\to Y$ is limit computable and uniformly continuous,
then $f$ is computable relative to $0''$.
\end{theorem}
\begin{proof}
We consider the computable metric spaces $(X,d_X,\alpha)$ and $(Y,d_Y)$.
As in the proof of Theorem~\ref{thm:limit-halting} we obtain that $f\circ\alpha:\IN\to Y$ is computable relative to $0'$.
Additionally, the set 
\[A:=\{(n,k)\in\IN^2:(\exists i,j\in\IN)(d_X(\alpha(i),\alpha(j))<2^{-n}\mbox{ and }d_Y(f\alpha(i),f\alpha(j))>2^{-k})\}\]
is c.e.\ relative to $0'$ and hence computable relative to $0''$. Since $f$ is uniformly continuous, it follows that for each $k\in\IN$ there are only finitely many $n\in\IN$ with $(n,k)\in A$.
Hence, for each $k\in\IN$ we can compute some $m(k)$ with $(m(k),k+1)\not\in A$ with the help of $0''$
and $m:\IN\to\IN$ is a uniform modulus of continuity of $f$ (i.e., at every $x\in X$).
We can proceed as in the proof of Theorem~\ref{thm:limit-halting} to show that $f$ is computable relative to $0''$.
\end{proof}

It is clear that this theorem also holds for compact computable metric spaces $X$ and continuous and limit computable $f$, since any such $f$ is automatically uniformly continuous. 

\begin{corollary}[Limit computability and uniform continuity]
\label{cor:limit-uniform-continuity}
Let $X,Y$ be computable metric spaces and let $X$ be compact. If $f:X\to Y$ is limit computable and continuous,
then $f$ is computable relative to $0''$.
\end{corollary}

We could also prove such a result for effectively locally compact $X$
and continuous and limit computable $f$, since it is sufficient to obtain the modulus of continuity for a suitable neighborhood of the input.

We will show that Theorem~\ref{thm:limit-uniform-continuity} gives the best possible condition and $0''$ cannot be improved to $0'$.
In the following we will see some examples of continuous functions that are limit computable but not computable relatively to the halting problem.
For these examples we use the set $\Fin:=\{n\in\IN:W_n$ finite$\}$, which is known to be 
$\sO{2}$--complete in the arithmetical hierarchy \cite[Theorem~4.3.2]{Soa16}, and hence it is not computable relative to the halting problem.
Its complement $\Inf:=\IN\setminus\Fin$ is $\pO{2}$--complete and hence not even c.e.\ relative to the halting problem.
By $|A|$ we denote the \emph{cardinality} of a set $A$.

\begin{proposition}
\label{prop:cantor-real-counter}
There exists a uniformly continuous function $f:2^\IN\to\IR$ that is computable with finitely many mind changes and hence low and limit computable, 
but that is not computable relative to the halting problem.
\end{proposition}
\begin{proof}
We define $f:2^\IN\to\IR$ by
\[f(p):=\left\{\begin{array}{ll}
2^{-n} & \mbox{if $0^n1^{|W_n|+1}0\prefix p$ for some $n\in\Fin$}\\
0      & \mbox{otherwise}
\end{array}\right..\]
Then $f$ is continuous and hence uniformly continuous, since $2^\IN$ is compact.
We use a fixed computable enumeration $e_n:\IN\to\IN$ of $W_n$, i.e., $\range(e_n)=W_n$.
The following algorithm shows that $f$ is computable with finitely many mind changes:
\begin{enumerate}
\item We start producing the default output $0$ as long as no other value is determined in the next step.
\item As soon as a prefix of the input $p$ of form $0^n1^{k+1}0w$ with $w\in\{0,1\}^m$ is known, we check if there are exactly $k$ 
        different values among the numbers 
        $e_n(0),...,e_n(m)$. If so, then the output is replaced by $2^{-n}$,
      otherwise by $0$.
\end{enumerate}
This algorithm describes a computation with finitely many mind changes, since for any fixed input $p$,
the output value changes at most two times.
In particular, $f$ is low and limit computable.
Let us assume that $f$ is computable relative to the halting problem. Then by Corollary~\ref{cor:uniform-metric} 
there exists a computable sequence $(f_k)_{k\in\IN}$ of computable functions 
$f_k:2^\IN\to\IR$ such that $\lim_{k\to\infty}||f_k-f||=0$.
Since the sequence $(K_n)_{n\in\IN}$ with $K_n:=0^n12^\IN$ is a computable sequence of computably compact sets, 
it follows that $(f_k(K_n))_{\langle k,n\rangle\in\IN}$ is a computable sequence of computably compact sets too \cite{Pau16}.
Since the maximum is computable on computably compact sets by \cite[Lemma~5.2.6]{Wei00}, 
it follows that $(\max f_k(K_n))_{k\in\IN}$ is a computable sequence, and it converges
for each $n\in\IN$ to $\max f(K_n)$. Hence, by
\[s_n:=\lim_{k\to\infty}\max f_k(K_n)=\max f(0^n12^\IN)=\left\{\begin{array}{ll}
2^{-n} & \mbox{if $n\in\Fin$}\\
0      & \mbox{otherwise}
\end{array}\right.\]
we define a sequence of numbers $(s_n)_{n\in\IN}$ that is limit computable and hence
computable relative to the halting problem by Corollary~\ref{cor:sequences}.
But this would imply that $\Fin$ is computable relative to the halting problem. Contradiction!
\end{proof}

We can transfer this construction from Cantor space to the unit interval $[0,1]$.

\begin{proposition}
\label{prop:unit-real-counter}
There exists a uniformly continuous function $f:[0,1]\to\IR$ that is computable with finitely many mind changes and hence low and limit computable, 
but that is not computable relative to the halting problem.
\end{proposition}
\begin{proof}
We use the computable ``triangle'' function 
\[\Delta:\IR\to\IR,x\mapsto\left\{\begin{array}{ll}
1-|x| & \mbox{if $-1\leq x\leq 1$}\\
0 & \mbox{otherwise}
\end{array}\right.\]
This function has the extreme values $f(-1)=f(1)=0$ and $f(0)=1$ and interpolates linearly between them otherwise.
Outside of the interval $(-1,1)$ the function $f$ is identically zero.
We define a computable double sequence $(\Delta_{n,k})_{n,k\in\IN}$ of triangle functions 
$\Delta_{n,k}:[0,1]\to\IR$ scaled to the interval $I_{n,k}:=(2^{-n-1},2^{-n-1}+2^{-n-k-1})$ by
\[\Delta_{n,k}(x):=\Delta(2^{n+k+2}(x-2^{-n-1})-1)\]
The function $\Delta_{n,k}$ is zero outside of $K_n:=[2^{-n-1},2^{-n}]$, in fact even outside of the interval
$I_{n,k}$. 
With growing $k$ the triangle of $\Delta_{n,k}$ 
is compressed further on the left hand side of $K_n$.
We define a continuous function $f:[0,1]\to\IR$ by
\[f(x):=\sum_{n\in\Fin}2^{-n}\Delta_{n,|W_n|}(x).\]
Similarly as in the proof of Proposition~\ref{prop:cantor-real-counter} one can show that $f$ is computable with finitely many mind changes
and not computable relative to the halting problem. For the latter part one considers the values
$s_n:=\max f(K_n)$. For the former part we proceed as follows with a given input $x\in[0,1]$:
\begin{enumerate}
\item As long as we cannot exclude that $x=2^{-n}$ for some $n\in\IN$ we produce the value $0$ as output.
\item As soon as we detect that $x\in(2^{-n-1},2^{-n})$ we consider the enumeration $e_n(0),e_n(1),...$
        of $W_n$.
\item Whenever we find a new number $k$ of distinct values in this enumeration, then we produce the value $2^{-n}\Delta_{n,k}(x)$ as output. 
\end{enumerate}
At any point $x$ at most a finite number of mind changes is required since either the value $k$ stabilizes or $\Delta_{n,k}(x)=0$
for all sufficiently large $k$.
\end{proof}

We note that Proposition~\ref{prop:unit-real-counter} cannot be strengthened such that the function $f$ is additionally continuously differentiable 
(see Corollary~\ref{cor:continuous-differentiable}). However, this can be achieved by using smooth versions of the triangle function $\Delta$
that are distributed all over the entire real line.

\begin{proposition}
\label{prop:real-real-counter}
There exists a function $f:\IR\to\IR$ that is infinitely often differentiable, computable with finitely many mind changes and hence low and limit computable, 
but that is not computable relative to the halting problem.
\end{proposition}
\begin{proof}
We use the computable ``bump'' function $\Delta:\IR\to\IR$
\[\Delta(x):=\left\{\begin{array}{ll}
e^{\frac{1}{1-x^2}} & \mbox{if $|x|<1$}\\
0                       & \mbox{otherwise}
\end{array}\right.,\]
which is infinitely often continuously differentiable, computable and zero outside of $(-1,1)$.
We define a computable double sequence $(\Delta_{n,k})_{n,k\in\IN}$ of bump functions 
$\Delta_{n,k}:\IR\to\IR$ scaled to the interval $I_{n,k}:=(n,n+2^{-k-1})$ by
\[\Delta_{n,k}(x):=\Delta(2^{k+2}(x-n)-1)\]
The function $\Delta_{n,k}$ is zero outside of $K_n:=[n,n+1]$, in fact even outside of the interval
$I_{n,k}$. 
With growing $k$ the triangle of $\Delta_{n,k}$ 
is compressed further on the left hand side of $K_n$.
We define $f:\IR\to\IR$ by
\[f(x):=\sum_{n\in\Fin}\Delta_{n,|W_n|}(x).\]
Similarly as in the proof of Proposition~\ref{prop:unit-real-counter} one can show that $f$ is computable with finitely many mind changes
and not computable relative to the halting problem. Additionally, $f$ is infinitely often differentiable, since the $\Delta_{n,k}$ are so.
\end{proof}

We can also produce a similar counterexample of type $f:2^\IN\to\IS$, where 
$\IS$ denotes \emph{Sierpi\'nski space} $\IS=\{0,1\}$ that is represented by $\delta_\IS(p)=0:\iff p=\widehat{0}$.

\begin{proposition}
\label{prop:cantor-sierp-counter}
There exists a continuous function $f:2^\IN\to\IS$ that is computable with finitely many mind changes and hence low and limit computable, but that is not computable relative to the halting problem.
\end{proposition}
\begin{proof}
We define $f:2^\IN\to\IS$ by
\[f(p):=\left\{\begin{array}{ll}
0 & \mbox{if $p=\widehat{0}$}\\
0 & \mbox{if $0^n1^{|W_n|+1}0\prefix p$ for some $n\in\Fin$}\\
1 & \mbox{otherwise}
\end{array}\right..\]
The sets $A_n:=0^n1^{|W_n|+1}02^\IN$ are clopen for every $n\in\Fin$ and 
the set $A:=\{\widehat{0}\}\cup\bigcup_{n\in\Fin}A_n$ is closed.
Hence $f$ is the characteristic function $\chi_U$ of the open set $U=2^\IN\setminus A$ and hence continuous.
The following algorithm shows that $f$ is computable with finitely many mind changes:
\begin{enumerate}
\item We start producing the default output $0$ as long as no other value is determined in the next step.
\item As soon as a prefix of the input $p$ of form $0^n1^{k+1}0^{m+1}$ is known, we check if there are exactly $k$ 
        different values among the numbers 
        $e_n(0),...,e_n(m)$. If so, then the output is replaced by $0$,
      otherwise by $1$.
\end{enumerate}
This algorithm describes a computation with finitely many mind changes, since for any fixed input $p$,
the output value changes at most three times.
In particular, $f$ is limit computable.

Let us assume that $f=\chi_U$ is also computable relative to the halting problem. 
We note that the map $\OO(2^\IN)\to\CC(2^\IN,\IS),V\mapsto\chi_V$ is a computable isomorphism,
and hence $U$ is a computable point in $\OO(2^\IN)'$.
The map $\forall_{K}:\OO(2^\IN)\to\IS$ with $\forall_K(V)=1\iff K\In V$ is computable for every
computably compact set $K\In2^\IN$, and the map $K\mapsto\forall_K$ is a computable isomorphism 
between the space of compact subsets of $2^\IN$ and $\CC(\OO(2^\IN),\IS)$~\cite{Pau16}. 
Since $(K_n)_{n\in\IN}$ with $K_n:=0^n12^\IN$ is a computable sequence
of computably compact sets, we obtain that $\forall_{K_n}$ is computable uniformly in $n$.
Hence $n\mapsto\forall_{K_n}(U)$ is computable relatively to the halting problem by Corollary~\ref{cor:sequences}.
This is a contradiction since
\[\forall_{K_n}(U)=1\iff 0^n12^\IN\In U\iff n\in\Inf\]
and $\Inf$ is not c.e.\ relative to the halting problem.
\end{proof}

Since a set $U\In2^\IN$ is c.e.\ open (relative to some oracle $q$) if and only if $\chi_U:2^\IN\to\IS$ is computable (relative to $q$) and $F\In2^\IN$
is a \emph{computable $\sO{2}$--set} in the Borel hierarchy (or, equivalently, a \emph{computable $F_\sigma$--set}) if and only 
if $\chi_F:2^\IN\to\IS'$ is computable \cite[Proposition~26]{PdB13},
the previous proposition can also be rephrased in the following way.

\begin{corollary}
\label{cor:open-set}
There is a computable $\sO{2}$--set $U\In2^\IN$ in the Borel hierarchy that is open, but not c.e.\ open relative to the halting problem.
\end{corollary}

A similar result for $\IR^n$ was proved by Ziegler~\cite[Theorem~4.4.~a)]{Zie07c}.
There is even a much stronger counter example for $2^\IN$.
There is a $\pO{2}$--set $A\In2^\IN$ that is a singleton (its only member being the $\omega$--jump of the empty set)
such that this set is closed but not a $\pO{1}$--set relative to $0^{(n)}$ for any $n\in\IN$ (see \cite[Proposition~1.8.62 and Exercise~1.8.67]{Nie09}, \cite[\S~15.1 Theorem~XII]{Rog87} and \cite{KT55}).

\begin{proposition}[Kuzn\'ecov and Traht\'enbrot 1955]
There is a $\sO{2}$--set $A\In2^\IN$ with a singleton complement that is not c.e.\ open relative to $0^{(n)}$ for any $n\in\IN$.
\end{proposition}

Further results along these lines can be found in~\cite{Cik75}.
If we re-translate this back to functions, then we obtain the following corollary.

\begin{corollary}
\label{cor:cantor-sierp-n-jump}
There is a continuous and limit computable $f:2^\IN\to\IS$ that is not computable relative to $0^{(n)}$ for any $n\in\IN$.
\end{corollary}

This also shows that Corollary~\ref{cor:limit-uniform-continuity} cannot be generalized to arbitrary spaces $Y$
and that the function $f$ in Corollary~\ref{cor:cantor-sierp-n-jump} cannot have a realizer $F:2^\IN\to\IN^\IN$
that is simultaneously continuous and limit computable, even though it has separate realizers with each of these
properties individually. 

There are some classes of functions for which the condition on the modulus of continuity in Theorem~\ref{thm:limit-halting} is
automatically satisfied. The first example is the class of H\"older continuous functions.
Let $(X,d_X)$ and $(Y,d_Y)$ be metric spaces. Then $f:X\to Y$ is called \emph{H\"older continuous} with constant $L>0$ and exponent $\alpha\in(0,1]$ if
\[d_Y(f(x),f(y))\leq L d_X(x,y)^\alpha\]
holds for all $x,y\in X$. 
A function is called \emph{Lipschitz continuous} with constant $L>0$ if it is
H\"older continuous with constant $L$ and exponent $\alpha=1$.
Since every H\"older continuous function automatically has a computable constant $L$ (perhaps slightly larger than necessary) and a computable exponent $\alpha$ (perhaps slightly smaller than necessary and correct at least for small distances $d_X(x,y)<1$),
it follows that it automatically has a computable global modulus of continuity. 
Hence we obtain the following corollary of Theorem~\ref{thm:limit-halting}.

\begin{corollary}[H\"older continuity]
\label{cor:Hoelder}
Let $X$ and $Y$ be computable metric spaces, and let $f:X\to Y$ be H\"older continuous.
Then $f$ is limit computable if and only if $f$ is computable relative to the halting problem.
\end{corollary}

Restricted to the set of H\"older continuous functions with a fixed Lipschitz constant
$L$ and a fixed exponent $\alpha$ the map in Corollary~\ref{cor:limit-halting} is even a computable 
isomorphism, as the proof of ``(2)$\TO$(1)'' of Theorem~\ref{thm:limit-halting} is uniform.
We note that even restricted to Lipschitz continuous functions the notions of a low function
and a function that is computable relative to a low oracle are incomparable by Proposition~\ref{prop:low-relative-low}.

In the table in Figure~\ref{tab:limit-oracle} we summarize which oracle classes are sufficient
to compute a limit computable function $f:X\to Y$ on complete computable metric spaces
that satisfies extra continuity assumptions.

\begin{figure}[htb]
\begin{tabular}{l|c}
$f$ limit computable and & oracle class\\\hline
&\\[-0.3cm]
continuous & $\dI{1}$\\[0.1cm]
uniformly continuous & $\sO{2}$\\[0.1cm]
Lipschitz or H\"older continuous & $\sO{1}$
\end{tabular}
\caption{Oracle classes sufficient to compute $f:X\to Y$.}
\label{tab:limit-oracle}
\end{figure}

Since every continuously differentiable function $f:[0,1]\to\IR$ is Lipschitz continuous with Lipschitz constant $\max f'[0,1]$,
we obtain the following conclusion (here it is important that $[0,1]$ is compact, and the conclusion does not hold 
for functions of type $f:\IR\to\IR$ by Proposition~\ref{prop:real-real-counter}).

\begin{corollary}[Continuously differentiable functions]
\label{cor:continuous-differentiable}
Let $f:[0,1]\to\IR$ be continuously differentiable. Then $f$ is limit computable if and only if $f$ is computable relative to the halting problem.
\end{corollary}

We can also draw some conclusions on linear operators.
We recall that a computable normed space is just a separable normed
space such that the induced metric space is computable and makes the algebraic operations computable.
By the Borel graph theorem of Schwartz~\cite{Sch66b} (see also \cite[Corollary~II.10.4]{AT05}) every Borel measurable
linear operator $T:X\to Y$ from a Banach space $X$ to a normed space
$Y$ is continuous, i.e., bounded.
Using Theorem~\ref{thm:Borel-continuity} we obtain the following
conclusion on effectively Borel measurable operators.

\begin{corollary}[Borel measurable linear functions]
\label{cor:linear-Borel}
Let $X$ be a computable Banach space, let $Y$ be computable normed spaces, and let $T:X\to Y$ be linear. Then $T$ is effectively Borel measurable if and only if
it is computable relative to a hyperarithmetical oracle (i.e., a $\Delta^1_1$ oracle). 
\end{corollary}

Every linear bounded operator is automatically Lipschitz continuous. 
In the special case of limit computability Corollary~\ref{cor:Hoelder} yields the following.

\begin{corollary}[Limit computable linear functions]
\label{cor:linear}
Let $X$ be a computable Banach space, $Y$ be computable normed spaces, and let $T:X\to Y$ be linear. 
Then $T$ is limit computable if and only if $T$ is computable relative to the halting problem.
\end{corollary}

\section{Applications}
\label{sec:applications}

In this section we discuss a number of simple applications of the tools that we have
developed in the previous sections. Some of the results that we derive are known,
but our techniques provide very simple proofs. Other results are new.

We start with considering distance functions.
For every subset $A\In X$ of a metric space $(X,d)$ we define the \emph{distance function} $d_A:X\to\IR,x\mapsto\inf_{y\in A}d(x,y)$.
Every distance function $d_A$ satisfies $|d_A(x)-d_A(y)|\leq d(x,y)$ for all $x,y\in X$, which means that every distance function
is Lipschitz continuous with Lipschitz constant $1$. 

\begin{corollary}[Distance functions]
\label{cor:distance}
Let $X$ be a computable metric space and $A\In X$.
The distance function $d_A:X\to\IR$ is limit computable if and only if it is computable relative to the halting problem.
\end{corollary}

We recall that a problem $f:X\to\IR$ is called \emph{lower semi-computable} if it is computable as a function $f:X\to\IR_<$
with the real numbers $\IR_<$ equipped with the lower Dedekind cut representation, and likewise $f$ is called \emph{upper semi-computable}
if it is computable as a function $f:X\to\IR_>$, where the real numbers $\IR_>$ are equipped with the upper Dedekind cut representation.
Since it is clear that $\id:\IR_<\to\IR$ and $\id:\IR_>\to\IR$ are limit computable~\cite[Proposition~3.7]{BG11a}, we obtain the following conclusion.

\begin{corollary}[Semi-computable functions]
\label{cor:semi-computable}
Every lower or upper semi-computable function $f:X\to\IR$ is limit computable.
\end{corollary}

By $f^{\circ n}$ we denote the $n$--fold composition of $f$ with itself, i.e., $f^{\circ 0}=\id$, $f^{\circ 1}=f$, $f^{\circ 2}=f\circ f$, etc.
It is easy to see that the \emph{Mandelbrot set} $M:=\{c\in\IC:(\forall n\in\IN)\;|f^{\circ n}_c(0)|\leq2\}$, which is defined using
the iteration function $f_c:\IC\to\IC,z\mapsto z^2+c$, is co-c.e.\ closed (see \cite[Exercise~5.1.32~(c)]{Wei00} and \cite{Her05}), 
and hence its distance function $d_M:\IC\to\IR$ is lower semi-computable \cite[Corollary~3.14~(2)]{BP03}.
As a corollary of Theorem~\ref{thm:limit-halting} and Corollaries~\ref{cor:distance} and \ref{cor:semi-computable} we obtain the following.

\begin{corollary}[Mandelbrot set]
\label{cor:Mandelbrot}
The distance function $d_M:\IC\to\IR$ of the Mandelbrot set $M\In\IC$ is computable relative to the halting problem.
\end{corollary}

Another application concerns the field of limit computable real numbers. By $(\IR^{(n)})_{\rm c}$ we denote the set of points that are computable in $\IR^{(n)}$, i.e., computable with respect to the $n$--th jump of the Cauchy representation.
We recall that a subfield of the reals is called \emph{real algebraically closed} if the real-valued zeros of all non-constant polynomials
with coefficients from the subfield are again in the subfield. 

\begin{proposition}[Field of real numbers]
The real numbers in $(\IR^{(n)})_{\rm c}$ form a real algebraically closed subfield
of $\IR$ for every $n\in\IN$.
\end{proposition}
\begin{proof}
We note the following facts for $n\in\IN$ and $a,b\in\IQ$ with $a<b$ (see \cite{Wei00}):
\begin{enumerate}
\item The algebraic operations $+,-,\cdot:\IR\times\IR\to\IR$ are computable.
\item Division $\div:\In\IR\times\IR\to\IR$ is computable.
\item The constants $0,1$ are computable real numbers.
\item The map $p:\IR^{n+1}\to\CC[a,b],(a_0,...,a_n)\mapsto(x\mapsto\sum_{i=0}^na_ix^i)$ is computable.
\item There is a computable $Z_{[a,b]}:\In\CC[a,b]\to\IR$ 
         such that $f(Z_{[a,b]}(f))=0$ for every continuous $f:[a,b]\to\IR$ that has exactly one zero (see \cite[Corollary~6.3.5]{Wei00}, where this was proved for $[a,b]=[0,1]$; it is straightforward to generalize the proof).
\end{enumerate}
It follows from (1)--(3) that $\IR_{\rm c}$ is a computable subfield of $\IR$ and from (4) and (5) that this field is real algebraically closed
as the zeros of all non-constant polynomials are isolated and can be computed if the coefficients are all computable. 
If we apply Proposition~\ref{prop:jump-endofunctor}, then we obtain that all the operations listed in (1)--(5) are also computable
when all spaces are replaced by their jumps. 
This implies that $(\IR')_{\rm c}$ is also a real algebraically closed subfield of the real numbers. 
Inductively, this property can be transferred to higher jumps.
\end{proof}

Zheng and Weihrauch~\cite[Proposition~7.6]{ZW01} also proved that $(\IR^{(n)})_{\rm c}$ is a subfield of the reals.
The fact that $\IR_{\rm c}$ is a real algebraically closed field was proved by Rice~\cite{Ric54} and Grzegorczyk~\cite{Grz55}.
Freund and Staiger~\cite{FS96} proved that $(\IR')_{\rm c}$ is a real algebraically closed field.

Finally, we want to discuss some simple applications of our methods to differentiable functions.
By $f':[0,1]\to\IR$ we denote the \emph{derivative} of a differentiable function $f:[0,1]\to\IR$. 
With very little effort we obtain the following result.

\begin{proposition}[Operator of differentiation]
\label{prop:differentiation-operator}
The following operation is computable:
$d:\In\CC([0,1],\IR)\to\CC([0,1],\IR'),f\mapsto f'$ with
$\dom(d):=\{f:[0,1]\to\IR:f$ differentiable$\}$.
\end{proposition}
\begin{proof}
If $f:[0,1]\to\IR$ is differentiable, then
\[f'(x)=\lim_{n\to\infty}\frac{f(x+(1-x)2^{-n})-f(x-x2^{-n})}{2^{-n}}\]
Using evaluation, currying and Theorem~\ref{thm:metric-limit-normal-form} we obtain that $d$ is computable.
\end{proof}

Using Theorem~\ref{thm:characterization-genericity-metric} we obtain the following result
of Kuyper and Terwijn~\cite[Theorem~4.3]{KT14}.

\begin{corollary}[Kuyper and Terwijn 2014]
\label{cor:derivatives}
The derivative $f':[0,1]\to\IR$ of every differentiable computable function $f:[0,1]\to\IR$
is limit computable and continuous at all $1$--generic points $x\in[0,1]$.
\end{corollary}

In fact, Kuyper and Terwijn obtained an even stronger result that characterizes $1$--generics in terms
of derivatives~\cite[Theorem~5.2]{KT14}. Our point here is not to strengthen these results, but
to illustrate the applicability of the tools that we have provided in this article. 

If $f:[0,1]\to\IR$ is even continuously differentiable then the limit in the proof of Proposition~\ref{prop:differentiation-operator}
can be seen to be a uniform limit. This leads to the following result of von Stein~\cite{Ste89}.
By $\CC^1$ we denote the \emph{set of continuously differentiable} functions $f:[0,1]\to\IR$.

\begin{proposition}[von Stein 1989]
\label{prop:continuous-differentiation-operator}
$d|_{\CC^1}:\In\CC([0,1],\IR)\to\CC([0,1],\IR)',f\mapsto f'$ is computable.
\end{proposition}
\begin{proof}
Let $f:[0,1]\to\IR$ be continuously differentiable. Then $F:[0,1]^2\to\IR$, defined by
\[F(x,y):=\left\{\begin{array}{ll}
f'(x)-\frac{f(x)-f(y)}{x-y} & \mbox{if $x\not=y$}\\
0 & \mbox{otherwise}
\end{array}\right.\]
is continuous, and since $[0,1]^2$ is compact, $F$ is even uniformly continuous.
Let $\varepsilon>0$. Since $F$ vanishes on the diagonal,
there exists $\delta>0$ such that $|x-y|<\delta\TO|F(x,y)|<\varepsilon$ for all $x,y\in[0,1]^2$.
The values $x_n:=x+(1-x)2^{-n}$ and $y_n:=x-x2^{-n}$ satisfy $|x_n-y_n|=2^{-n}$.
Hence there is a $k\in\IN$ such that $|x_n-y_n|<\delta$ for all $n\geq k$.
Altogether, this shows that $f_n:[0,1]\to\IR$ with
\[f_n(x)=\frac{f(x+(1-x)2^{-n})-f(x-x2^{-n})}{2^{-n}}.\]
satisfies $||f_n-f'||<\varepsilon$ for $n\geq k$.
Hence, we have the uniform limit $\lim_{n\to\infty}f_n=f'$.
The sequence $(f_n)_{n\in\IN}$ in $\CC([0,1],\IR)$ can be computed from $f$ using evaluation and currying.
It follows that $d|_{\CC^1}$ is computable in the given way by Theorem~\ref{thm:metric-limit-normal-form}.
\end{proof}

This result can also be phrased such that $d|_{\CC^1}:\In\CC([0,1],\IR)\to\CC([0,1],\IR),f\mapsto f'$ is limit computable. 
In fact, von Stein even proved that $d|_{\CC^1}$ with this type is Weihrauch equivalent to $\lim$, which yields for instance
Myhill's result~\cite{Myh71} that there is a continuously differentiable computable function $f:[0,1]\to\IR$ with
a non-computable derivative $f':[0,1]\to\IR$ (see \cite{Bra99}). 
Here we rather aim for the following corollary that was originally proved by Ho (announced in \cite[Theorem~1.3]{Ho96} and 
proved in \cite[Theorem 19]{Ho99}).

\begin{corollary}[Ho 1999]
\label{cor:continuous-derivatives}
The derivative $f':[0,1]\to\IR$ of every continuously differentiable computable function $f:[0,1]\to\IR$
is computable relative to the halting problem.
\end{corollary}

With the following proposition we generalize the observation \cite[Corollary~9.10~(2)]{BHK16} to
all computable metric spaces. 

\begin{proposition}[C.e.\ comeager sets]
\label{prop:ce-comeager}
Let $X$ be a computable metric space and let $(A_n)_{n\in\IN}$ be a computable sequence
of co-c.e.\ closed subsets $A_n\In X$ that are all nowhere dense. Then
$A:=X\setminus\bigcup_{n\in\IN}A_n$ is a comeager set that contains all
$1$--generic points $x\in X$.
\end{proposition}
\begin{proof}
If the $A_n\In X$ are nowhere dense (have non-empty interior) then $A_n=\partial A_n=\partial A_n^{\rm c}$. Hence 
$X\setminus\bigcup_{n\in\IN}\partial U_n\In X\setminus\bigcup_{n\in\IN}\partial A_n^{\rm c}=A$.
\end{proof}

We can also consider $1$--generic points in the space $\CC[0,1]$ of continuous functions.
In \cite{Bra01a} it was proved that the set of somewhere differentiable functions is included in
 $\bigcup_{n=0}^\infty D_n$, where $(D_n)_{n\in\IN}$ is a computable sequence of nowhere dense co-c.e.\ 
subsets $D_n\In\CC[0,1]$, defined by
$D_n:=\left\{f\in\CC[0,1]:(\exists t\in[0,1])(\forall h\in\IR\setminus\{0\})\left|\frac{f(t+h)-f(t)}{h}\right|\leq n\right\}$.
Hence, we obtain the following observation.

\begin{corollary}[Nowhere differentiability of $1$--generic functions]
\label{cor:nowhere-differentiable}
Every continuous function $f:[0,1]\to\IR$ that is $1$--generic as a point in $\CC[0,1]$
is nowhere differentiable.
\end{corollary}

The mere property that every $1$--generic continuous $f:[0,1]\to\IR$ is not differentiable could also be deduced from 
the fact that the operator of differentiation $d$ is not continuous at any point $f$.

These examples just serve as illustrations that a careful analysis of the notions of limit computability and computability
with respect to the halting problem is useful for applications in analysis. In particular the concepts based on the 
Galois connection between limit and Turing jumps yield
very transparent and simple proofs.

\bibliographystyle{plain}
\bibliography{C:/Users/\input{"|kpsewhich -var-value USERNAME"}/Dropbox/Bibliography/lit}

\section*{Acknowledgments}

The author would like to thank Alex Simpson for asking a question during the Logic Colloquium 2007 in Wroc{\l}aw that
inspired the study of the Galois connection between Turing jumps and limits.
The author also acknowledges invitations by Christine Gassner and Alberto Marcone to Hiddensee and Udine, respectively, in 2017,
where the results of section~\ref{sec:computability-halting-problem} were worked out.
Last not least, a discussion with Russell Miller in Dagstuhl 2017 motivated the author to include the examples on differentiable
functions, and a discussion with Arno Pauly in Oberwolfach 2018 has inspired the author to include Theorem~\ref{thm:jump-metric}.
The author is also grateful for helpful comments by two anonymous referees that helped to improve the final
version of this article.

\end{document}